\documentclass{amsart}

\newtheorem{theorem}{Theorem}[section]
\newtheorem{lemma}[theorem]{Lemma}
\newtheorem{remark}[theorem]{Remark}
\newtheorem{proposition}[theorem]{Proposition}
\newtheorem{corollary}[theorem]{Corollary}

\newtheorem{example}[theorem]{Example}

\usepackage{amsmath, amsthm, amssymb}
\usepackage{mathtools}
\usepackage{hyperref}
\usepackage{enumerate}
\usepackage{nicefrac}
\usepackage{tikz}

\newcommand{\cemph}[1]{\emph{\color{red}#1}}

\newcommand\R{\mathbb{R}}
\newcommand\B{\mathbb{B}}
\newcommand\N{\mathbb{N}}
\newcommand\CK{\mathcal{K}}

\DeclareMathOperator{\conv}{conv}
\DeclareMathOperator{\aff}{aff}
\DeclareMathOperator{\lin}{lin}
\DeclareMathOperator{\inte}{int}
\DeclareMathOperator{\bd}{bd}
\DeclareMathOperator{\vol}{vol}

\newcommand\gauge[1]{\left\| #1 \right\|}

\newcommand\op[2]{#1 \vert #2}

\newcommand\jel{\mathcal{E}_J}
\newcommand\lel{\mathcal{E}_L}

\title[Volume Extremal k-Ellipsoids]{Tightening Inequalities on Volume Extremal $k$-Ellipsoids Using Asymmetry Measures}
\author[R.~Brandenberg]{Ren\'e Brandenberg}
\author[F.~Grundbacher]{Florian Grundbacher}
\keywords{John ellipsoid, Loewner ellipsoid, $k$-Ellipsoid, John decomposition, John asymmetry, Geometric inequalities under affinity, Banach-Mazur distance}
\subjclass[2020]{Primary 52A40; Secondary 52A38}
\date{December 31, 2024}

\begin{document}

\maketitle

\begin{abstract}
We consider two well-known problems: upper bounding the volume of lower dimensional ellipsoids contained in convex bodies given their John ellipsoid, and lower bounding the volume of ellipsoids containing projections of convex bodies given their Loewner ellipsoid. For the first problem, we use the John asymmetry to unify a tight upper bound for the general case by  Ball with a stronger inequality for symmetric convex bodies. We obtain an inequality that is tight for most asymmetry values in large dimensions and an even stronger inequality in the planar case that is always best possible. In contrast, we show for the second problem  an inequality that is tight for bodies of any asymmetry, including cross-polytopes, parallelotopes, and (in almost all cases) simplices. Finally, we derive some consequences for the width-circumradius- and diameter-inradius-ratios when optimized over affine transformations and show connections to the Banach-Mazur distance.
\end{abstract}

\section{Introduction and Notation}

Inequalities between geometric functionals have long been a central field of study in convex geometry. Following Jung's inequality on the Euclidean diameter-circumradius-ratio \cite{Ju},
especially results comparing different radial functionals have gained significant attention.
They include classics such as Steinhagen's inequality for Euclidean spaces \cite{St},
Bohnenblust and Leichtweiss' analogs of the mentioned inequalities for general normed spaces \cite{Bo,Le}, and other ground-laying results involving volume functionals \cite{Bi,Pal} or affinity \cite{Be,Jo1,Jo}. Research in this direction continues to this day
(cf.~\cite{Al,BeHe,BrGorRD,BrKoSharp,He}).
While these inequalities form fundamental insights for themselves, they are also important for applications in other mathematical contexts and approximation algorithms.

In the literature, significant effort has been put into comparing families of inner and outer radii of convex bodies and computing them explicitly for regular polytopes (cf.~\cite{BeHe,Br,GrKlRad,He} and the references therein).
Of particular interest to our work are the outer and inner $k$-radii.
The \cemph{outer $k$-radius} $R_k$ is defined as the smallest radius of a Euclidean $k$-ball containing the projection of a body onto some linear $k$-space.
The \cemph{inner $k$-radius} $r_k$ is the largest radius of a Euclidean $k$-ball contained in a body. Other variants of these radii have been considered as well, though we shall not consider them here.
Note that $R \coloneqq R_n$ is simply the \cemph{circumradius}, $w \coloneqq 2 R_1$ the \cemph{minimal width}, $r \coloneqq r_n$ the \cemph{inradius}, and $D \coloneqq 2 r_1$ the \cemph{diameter}.
All of these radii have been computed explicitly for regular cross-polytopes, cubes, and regular simplices in dependence of their in- or circumradius \cite{Br}.

In the present work, we consider variants of these radii-comparison-problems where, in addition, affine transformations are involved.

In the first problem, we give a lower bound on the volume of $k$-ellipsoids containing the projection of a convex body $K$ onto some linear $k$-space in dependence of the \cemph{Loewner ellipsoid} $\lel(K)$ of $K$, i.e., its volume minimal circumscribed $n$-ellipsoid. In the second problem, we study upper bounds on the volume of $k$-ellipsoids contained in $K$ in dependence of its \cemph{John ellipsoid} $\jel(K)$, which is the volume maximal $n$-ellipsoid inscribed in $K$.
We refer to these problems as \emph{outer-} or \emph{inner-$k$-ellipsoid problem}, respectively.
Finally, we derive bounds for different ratios of radial functionals when affine transformations of the underlying bodies are allowed.

We write $\CK^n$ for the family of \cemph{(convex) bodies} in $\R^n$, i.e., compact convex sets with non-empty interior.
The \cemph{Euclidean unit ball} is denoted by $\B^n$
and the $k$-dimensional \cemph{volume} by $\vol_k$.
For $X,Y \subset \R^n$, their \cemph{Minkowski sum} is given by $X+Y \coloneqq \{ x+y : x \in X, y \in Y \}$.
Given $t \in \R^n$ and $\rho \in \R$, the \cemph{$t$-translation} and \cemph{$\rho$-dilatation} of $X$ are defined by $X+t \coloneqq t+X \coloneqq \{t\}+X$ and $\rho X \coloneqq \{ \rho x : x \in X \}$, respectively.
We abbreviate $-X = (-1) X$ and
$[m] \coloneqq \{ i \in \N : i \leq m \}$ for $m \in \N_0$.

The inner-$k$-ellipsoid problem goes back to a result by Ball stemming from his study of the inner $k$-radii of regular simplices.
He solved the general case of this problem by proving the following result \cite[Proposition~4]{Ball}.

\begin{proposition} \label{prop:ball_thm}
Let $k \in [n-1]$, $K \in \CK^n$ with $\jel(K) = \B^n$,  and $E \subset K$ be a $k$-ellipsoid. Then
\[
    \vol_k(E) \leq \vol_k \left( \sqrt{ \frac{n (n+1)}{k (k+1)} } \, \B^k \right),
\]
with equality for the inscribed Euclidean $k$-balls of $k$-faces of regular simplices.
\end{proposition}

The duality of the two problems above is apparent from their statements.
While this duality might suggest that the best possible bounds should correspond to each other accordingly,
it turns out that this is not the case.
For the outer-$k$-ellipsoid problem, we prove a bound that is even stronger than what one would obtain by taking the reciprocal of the radius of the $k$-ball in the above inequality in Section~\ref{sec:outer}.
We denote the \cemph{orthogonal projection} of a vector $x \in \R^n$ and a set $X \subset \R^n$
onto a linear $k$-space $F$ by $\op{x}{F}$ and $\op{X}{F}$, respectively.

\begin{theorem} \label{thm:outer}
Let $k \in [n-1]$, $K \in \CK^n$ with $\lel(K) = \B^n$, $F \subset \R^n$ a linear $k$-space, and $E \subset F$ a $k$-ellipsoid with $\op{K}{F} \subset E$. Then
\[
    \vol_k(E)
    \geq \vol_k \left( \sqrt{\frac{k}{n}} \, \B^k \right),
\]
with equality if and only if 
$E = \sqrt{\frac{k}{n}} \, (\op{\B^n}{F})$.
Moreover, there exists a family of convex bodies that is continuous with respect to the Hausdorff distance, includes a regular cross-polyope and a cube, and can for $n$ odd or $k \notin \{1,n-1\}$ even be extended to include a regular simplex, such that equality holds for appropriately chosen $k$-balls for any body in the family.
\end{theorem}

Let us now discuss the duality of the two problems above in more detail to explain the discrepancy between Proposition~\ref{prop:ball_thm} and Theorem~\ref{thm:outer}.
Denoting the \cemph{polar} of a set $X \subset \R^n$ by $X^\circ \coloneqq \{ a \in \R^n : a^T x \leq 1 \text{ for all } x \in X \}$,
it is well-known for a convex body $K \in \CK^n$ containing the origin in its interior that $K^\circ \in \CK^n$,
and $\jel(K) = \B^n$ if and only if $\lel(K^\circ) = \B^n$ (see Proposition~\ref{prop:john}).
Moreover, an infinite cylinder $Z$ with ellipsoidal base $E$ through the origin contains $K^\circ$ if and only if $K$ contains the polar of $E$ with respect to its \cemph{linear hull} $\lin(E)$.
Since $K^\circ \subset Z$ is equivalent to $\op{K^\circ}{\lin(E)} \subset E$, our two major problems appear to be dual -- at least at first glance.
However, the duality between the two problems is not perfect.
Only $k$-ellipsoids that contain the origin in their relative interior have an infinite cylinder with $k$-ellipsoidal base as their polar. Hence, the duality does not cover any ellipsoid
not containing the origin in its relative interior since there is no corresponding infinite cylinder that we could easily associate it with.
This gap in the duality dissolves if $K$ has to be \cemph{(centrally) symmetric}, i.e., $K-c = c-K$ for some $c \in \R^n$.
It is easy to see that in this case there always exists a volume maximal $k$-ellipsoid inscribed in $K$ with the same (symmetry) center as $K$.
Since the center of $K$ must be the origin if $\jel(K) = \B^n$,
Theorem~\ref{thm:outer} and the above described duality can be used to improve the inequality in Proposition~\ref{prop:ball_thm} for symmetric $K$.
The same bound has previously been established in \cite{Iv} for the restricted case of $k$-ellipsoids contained in sections of cubes, together with the corresponding special case of Theorem~\ref{thm:outer} for projections of regular cross-polytopes.
\begin{corollary} \label{cor:ball_sym_improv}
Let $k \in [n-1]$,
$K \in \CK^n$ symmetric with $\jel(K) = \B^n$,
and $E \subset K$ be a $k$-ellipsoid.
Then
\[
    \vol_k(E)
    \leq \vol_k \left( \sqrt{\frac{n}{k}} \, \B^k \right),
\]
with equality if $K$ is a regular cross-polytope or cube
and $E$ is chosen appropriately.
\end{corollary}

The case distinction
in the upper bounds for the inner-$k$-ellipsoid problem,
depending on whether symmetry is imposed or not,
is a common phenomenon for geometric inequalities.
In many such instances, the two results can be unified by involving an \textit{asymmetry measure}.
These measures express numerically how ``(a)symmetric'' a set is,
where the precise meaning of ``(a)symmetry'' should typically be adapted to the respective context.
A first systematic definition and study of asymmetry measures has been provided in \cite{Gr}.
Usually, they are required to be continuous (with respect to the Hausdorff metric) and attain one of their extremal values precisely for ``symmetric'' convex bodies.
In many cases (especially when talking about ``point-asymmetries''), the other extremal value is attained by simplices, which reflects the intuition that simplices are the ``most asymmetric'' convex bodies in many contexts.
Involving asymmetry measures in geometric inequalities often not only unifies results, but also streamlines proofs and provides additional information about the extreme cases of the inequalities.
Whenever asymmetry measures can be well computed/approximated, they even improve practical applications of geometric inequalities (cf.~\cite{BeFr,BrKoSharp}).

The probably most well-known example of an asymmetry measure of a convex body $K$ is the \cemph{Minkowski asymmetry} $s(K)$.
It is defined as the smallest dilatation factor for a homothet of $-K$ to cover $K$.
The Minkowski asymmetry has been used frequently to strengthen geometric inequalities
(cf.~\cite{AkBaGr,BeFr,BrGo,BrGojung,BrGorRD,BrKoSharp,GLMP})
and occasionally also as a tool to obtain stability results (cf.~\cite{Boro,BDG,Sch}).
We shall use the Minkowski asymmetry for our results in Section~\ref{sec:aff_ratios}, but our results concerning the inner-$k$-ellipsoid problem are based on an asymmetry measure that is more intrinsic to that problem.

A natural choice of the asymmetry measure for the inner-$k$-ellipsoid problem is the \cemph{John asymmetry},
which is closely related to the Minkowski asymmetry.
It has been used before to, e.g., sharpen John's theorem on the approximation of convex bodies by $n$-ellipsoids (see Proposition~\ref{prop:john_rounding}).
Denoting for a body $K$ the center of $\jel(K)$ by $c_J(K)$,
the John asymmetry of $K$ is defined by
\[
    s_J(K)
    \coloneqq \inf \{ \rho \geq 0 : K - c_J(K) \subset \rho (c_J(K) - K) \}.
\]
It is known that the John asymmetry shares some basic but important properties with the Minkowski asymmetry.
For example, both only take values between $1$ and $n$.
As expected of ``point-asymmetries'', the smaller value characterizes centrally symmetric bodies and the larger value is attained for simplices.
However, unlike for the Minkowski asymmetry, simplices are not the unique maximizers of the John asymmetry (see Appendix~\ref{app:sJ}).

Using the John asymmetry, we
unify the upper bounds in Proposition~\ref{prop:ball_thm} and
Corollary~\ref{cor:ball_sym_improv} in Section~\ref{sec:inner}.
We also obtain a characterization of the ellipsoids that reach the equality case,
which is used for the proof of some implications in Section~\ref{sec:aff_ratios}.
Note that there is a relatively large degree of freedom in the choice of $K$ for which there exists an ellipsoid that reaches the upper bound.
A similar problem occurs in the setting of Theorem~\ref{thm:outer}, so one should not expect to find a meaningful characterization of all the extremal bodies $K$ for either theorem.

\begin{theorem} \label{thm:inner}
Let $k \in [n-1]$, $K \in \CK^n$ with $\jel(K) = \B^n$, and $E \subset K$ a $k$-ellipsoid. Then
\[
    \vol_k(E)
    \leq \vol_k \left( \sqrt{\frac{n}{k} \min \left\{ \frac{s_J(K)+1}{2},\frac{n+1}{k+1} \right\} } \, \B^k \right),
\]
with equality if and only if
\begin{enumerate}[(i)]
\item $E$ is a Euclidean $k$-ball
of radius $\sqrt{\frac{n}{k} \min \{ \frac{s_J(K)+1}{2},\frac{n+1}{k+1} \} }$, and
\item the center $c$ of $E$ satisfies $E \subset \{ x \in \R^n : c^T x = n (\min \{ \frac{s_J(K)+1}{2},\frac{n+1}{k+1} \} - 1) \}$.
\end{enumerate}
Moreover, the inequality is best possible 
precisely for any prescribed 
$s_J(K) \notin (1,1+\frac{2}{n})$.
\end{theorem}

Writing $\gauge{\cdot}$ for the \cemph{Euclidean norm}, (ii) is equivalent to $\gauge{c}^2 = n (\min \{ \frac{s_J(K)+1}{2},\frac{n+1}{k+1} \} - 1)$ and,
in case $c \neq 0$, i.e., $s_J(K) \neq 1$, the line connecting the origin and $c$ being perpendicular to the \cemph{affine hull} $\aff(E)$.
In particular, the ellipsoid $E$ that can be contained in a convex body $K \in \CK^n$ with $\jel(K) = \B^n$ and $s_J(K) = s$ that reaches the upper bound is for any $s \notin (1,1+\frac{2}{n})$
unique up to rotation around the origin.

For the remaining range of asymmetries $s_J(K) \in (1,1+\frac{2}{n})$, we construct in Example~\ref{ex:inner_small_asym} a family of convex bodies that we believe to provide the extreme cases.
Together with Theorem~\ref{thm:inner}, they in particular show that the radius of the Euclidean $k$-ball in the best possible upper bound lies between $\sqrt{\frac{n}{k}}$ and $\sqrt{\frac{n+1}{k}}$.
The ratio of these values, as well as the length of the interval $(1,1+\frac{2}{n})$ where the best possible upper bound is unproven, decays with order $n^{-1}$ and in particular independently of $k$.

Although the gap in the tightness in Theorem~\ref{thm:inner} is for large $n$ rather narrow, it is in contrast comparatively large for small $n$.
Especially for $n=2$, where $1+\frac{2}{n} = n$, the inequality is tight only for the extremal John asymmetry values.
To remedy this, we devote special attention to the planar case in Section~\ref{sec:planar_diam}
and prove the strengthened inequality below,
which verifies at least for $n=2$ and $k=1$ that Example~\ref{ex:inner_small_asym} indeed provides the largest possible value.
Since $1$-ellipsoids are segments and their $1$-dimensional volume is simply the distance between their endpoints,
it is more natural to state the result in the language of diameters.
For $s \in [1,2]$ define 
\begin{equation} \label{eq:Ds}
    D_s
    \coloneqq \sqrt{s^2 + 5 + \sqrt{4 (2-s)^2 + (s^2-1)^2}}.
\end{equation}

\begin{theorem} \label{thm:planar_diam}
Let $K \in \CK^2$ with $\jel(K) = \B^2$.
Then
\[
    D(K)
    \leq D_{s_J(K)},
\]
and any $x,y \in K$ satisfy $\gauge{x-y} = D_{s_J(K)}$ if and only if
\begin{enumerate}[(i)]
\item $\gauge{x} = \gauge{y} = \sqrt{ \frac{1}{2} D_{s_J(K)}^2 - 2}$, and
\item $\gauge{\frac{x+y}{2}} = \sqrt{ \frac{1}{4} D_{s_J(K)}^2 - 2}$.
\end{enumerate}
Moreover, the inequality is best possible for any prescribed $s_J(K) \in [1,2]$.
\end{theorem}

For an application of our results,
we stay with $1$-ellipsoids in Section~\ref{sec:aff_ratios}
and return to problems of directly comparing radial functionals along the lines of Jung, Steinhagen, Bohnenblust, and Leichtweiss.
We consider the (to the best of our knowledge) new variants of bounding the width-inradius- and diameter-circumradius-ratios
when the first is maximized and the second minimized over all affine transformations of a convex body.
Since these problems do not intrinsically depend on Euclidean distances, we also consider them in general Minkowski spaces.
An upper bound for the first and a lower bound for the second ratio can be derived from \cite[Corollary~6.4]{BrKoSharp}.
Based on our results above, we obtain sharp bounds in the respective other directions, both in the Euclidean and the general case.
Lastly, we show how these results are connected to the extreme cases for some well-known bounds on the Banach-Mazur/Grünbaum distance to the Euclidean ball and, more generally, arbitrary convex bodies.

In closing the first section, let us collect some more notation used throughout.
For a set $X \subset \R^n$, we denote by $\conv(X)$
its \cemph{convex hull}.
We write $\bd(X)$ and $\inte(X)$ for the \cemph{boundary} and \cemph{interior} of $X$.
The \cemph{support function} of $X$ is given by $h_X \colon \R^n \to \R, h_X(a) \coloneqq \max \{ a^T x : x \in X \}$.
For $a \in \R^n \setminus \{0\}$ and $\beta \in \R$, we define the \cemph{halfspace}
$H_{(a,\beta)}^\leq \coloneqq \{ x \in \R^n : a^T x \leq \beta \}$.
The \cemph{orthogonal complement} of a linear subspace $F \subset \R^n$ is $F^\perp$.
We denote the closed \cemph{segment} connecting $x,y \in \R^n$ by $[x,y]$, where we replace the brackets with parentheses if we wish to exclude the respective endpoints from the segment.

Let us also state the well-known characterization of Loewner and John ellipsoids obtained by John in \cite{Jo},
which provides the basis for our results throughout the paper.
We write $I_n$ for the \cemph{identity matrix} in $\R^{n \times n}$.

\begin{proposition} \label{prop:john}
A convex body $K \in \CK^n$ satisfies $\jel(K) = \B^n$ (resp.~$\lel(K) = \B^n$) if and only if
\begin{enumerate}[(i)]
\item $\B^n \subset K$ (resp.~$K \subset \B^n$) and
\item there exist $u^1, \ldots, u^m \in \bd(K) \cap \bd(\B^n)$ as well as $\lambda_1, \ldots, \lambda_m > 0$ such that
\[
    \sum_{i \in [m]} \lambda_i u^i = 0
    \quad \text{and} \quad
    \sum_{i \in [m]} \lambda_i u^i (u^i)^T = I_n.
\]
\end{enumerate}
\end{proposition}

Obviously, (ii) also implies $\sum_{i \in [m]} \lambda_i = \mathrm{trace}(I_n) = n$.
We say that Euclidean unit vectors and weights satisfying the conditions in (ii) above form a \cemph{John decomposition} (for $K$)
and that a convex body $K$ is in \cemph{John position} (resp.~\cemph{Loewner position}) if $\jel(K) = \B^n$ (resp.~$\lel(K) = \B^n$).

\section{Minimal Ellipsoids Containing Projections}
\label{sec:outer}

In this section we prove Theorem~\ref{thm:outer}.
We begin with the claimed inequality
and afterwards show its tightness
by providing a family of convex bodies as described in the theorem.

\begin{lemma} \label{lem:outer}
Let $u^1, \ldots, u^m \in \bd(\B^n)$ and $\lambda_1, \ldots, \lambda_m > 0$ form a John decomposition.
If $F \subset \R^n$ a linear $k$-space and $E \subset F$ is a $k$-ellipsoid with $\op{u^i}{F} \in E$ for all $i \in [m]$, then
\[
    \vol_k(E)
    \geq \vol_k \left( \sqrt{\frac{k}{n}} \, \B^k \right),
\]
with equality if and only if
$E = \sqrt{\frac{k}{n}} \left( \op{\B^n}{F} \right)$.
\end{lemma}
\begin{proof}
Let $v^1, \ldots, v^k \in F$ be orthonormal vectors,
$c \in F$,
and $\mu_1, \ldots, \mu_k > 0$ such that
\[
    E
    = \left\{ x \in F : \sum_{j \in [k]} \mu_j ((x-c)^T v^j)^2 \leq \frac{k}{n} \right\}.
\]
It is well-known that
\[
    \vol_k(E)
    = \sqrt{ \prod_{j \in [k]} \frac{1}{\mu_j} } \, \vol_k \left( \sqrt{\frac{k}{n}} \, \B^k \right),
\]
so the claimed inequality is equivalent to
\[
    \prod_{j \in [k]} \mu_j
    \leq 1.
\]

We write $p^i \coloneqq \op{u^i}{F}$ for all $i \in [m]$.
Since $u^i = p^i + (\op{u^i}{F^\perp})$,
we have for all $j \in [k]$ that $(p^i)^T v^j = (u^i)^T v^j$.
Consequently, the assumption $p^i \in E$ implies
\[
    \sum_{j \in [k]} \mu_j ((u^i-c)^T v^j)^2
    = \sum_{j \in [k]} \mu_j ((p^i-c)^T v^j)^2
    \leq \frac{k}{n}.
\]
Since $\sum_{i \in [m]} \lambda_i = n$, it follows
\begin{align}
    k
    & \geq \sum_{i \in [m]} \lambda_i \left( \sum_{j \in [k]} \mu_j \left( (u^i-c)^T v^j \right)^2 \right) \nonumber
    \\
    & = \sum_{j \in [k]} \mu_j (v^j)^T \left( \sum_{i \in [m]} \lambda_i \left( u^i (u^i)^T - u^i c^T - c (u^i)^T + c c^T \right) \right) v^j. \label{eq:outer_ineq}
\end{align}
Because of the conditions about the $u^i$ and $\lambda_i$, the first term in the innermost brackets of \eqref{eq:outer_ineq} sums to $I_n$ and the second and third vanish.
Thus, we obtain
\begin{equation} \label{eq:outer_center_ineq}
    k \geq \sum_{j \in [k]} \mu_j + n \sum_{j \in [k]} \mu_j (c^T v^j)^2 \geq \sum_{j \in [k]} \mu_j.
\end{equation}
Hence, by the inequality between the geometric mean and arithmetic mean, it follows
\[
    \prod_{j \in [k]} \mu_j \leq \left( \sum_{j \in [k]} \frac{\mu_j}{k} \right)^k \leq 1,
\]
with equality if and only if $\mu_j = 1$ for all $j \in [k]$. In the equality case we also have equality from left to right in \eqref{eq:outer_center_ineq},
which implies $c \in \lin(\{ v^1, \ldots, v^k \})^\perp = F^\perp$. Since $c \in F$, this means $c = 0$. Altogether, the claimed equality condition follows.
\end{proof}

By Proposition~\ref{prop:john}, any convex body $K \in \CK^n$ with $\lel(K) = \B^n$ contains a set of vectors as in the above lemma,
so the inequality in Theorem~\ref{thm:outer} and its equality case are established by Lemma~\ref{lem:outer}.
For the rest of this section, we turn to its tightness.
Our starting point is the following proposition taken from \cite{Br}.

\begin{proposition} \label{prop:k-rad}
Let $k \in [n]$ and $P \in \CK^n$ be a regular cross-polytope, cube, or regular simplex. Then
\[
    \frac{R_k(P)}{R(P)} = \begin{cases}
        \frac{n+1}{n \sqrt{n+2}}, & \text{if $P$ is a simplex, $n$ even, $k=1$}, \\
        \frac{2n-1}{2n}, & \text{if $P$ is a simplex, $n$ even, $k=n-1$}, \\
        \sqrt{\frac{k}{n}}, & \text{else}.
    \end{cases}
\]
\end{proposition}

This proposition already shows that the inequality in Theorem~\ref{thm:outer} is best possible for any $n$ and $k$.
Surprisingly, in almost all cases, the inequality is tight both for some symmetric bodies as well as simplices.
As outlined in the introduction, these are the extreme cases for many reasonable asymmetry measures.
The example below shows that, unless $n$ is even and $k \in \{1,n-1\}$, the inequality is tight also for all intermediate asymmetry values for any asymmetry measure that is additionally continuous with respect to the Hausdorff distance on $\CK^n$.

\begin{example}
Let $k \in [n]$. Then there exists a function $K \colon [0,1] \to \CK^n$ that is continuous with respect to the Hausdorff distance
such that $K(0)$ is a cross-polytope,
$K(1)$ is a parallelotope,
and for every $t \in [0,1]$ we have $\lel(K(t)) = \B^n$ and $R_k(K(t)) = \sqrt{\frac{k}{n}}$.
If $n$ is odd or $k \notin \{1,n-1\}$, the function can be continuously extended to the interval $[0,2]$ such that the latter two properties
remain true for all $t \in [0,2]$ and $K(2)$ is a simplex.
\end{example}
\begin{proof}
Let $P_1, P_\infty, T \in \CK^n$ be a regular cross-polytope, a cube, and a regular simplex, respectively, such that the Loewner ellipsoid of all three polytopes
is $\B^n$. We may additionally assume that there exists a single linear $k$-space $F \subset \R^n$
such that the outer $k$-radius of all three polytopes is attained for their projections onto $F$. For any $t \in [0,2]$ we define
\[
    K(t) \coloneqq \begin{cases}
        \conv( P_1 \cup (2t P_\infty)), & \text{if $t \in [0,0.5]$},\\
        \conv( (2(1-t) P_1) \cup P_\infty), & \text{if $t \in [0.5,1]$}, \\
        \conv( P_\infty \cup (t-1) T), & \text{if $t \in [1,1.5]$, and} \\
        \conv( (2 (2-t) P_\infty) \cup T), & \text{if $t \in [1.5,2]$}.
    \end{cases}
\]
It is easy to see from the above construction that $K(t)$ is continuous in $t$ with $K(0) = P_1$, $K(1) = P_\infty$, and $K(2) = T$. 
Moreover, by the choice of $P_1$, $P_\infty$, and $T$, all $K(t)$ are subsets of $\B^n$ and contain at least one of the three explicitly given polytopes entirely. 
Thus, $\lel(K(t)) = \B^n$ for all $t \in [0,2]$. 

Lastly, we obtain from $\lel(K(t)) = \B^n$ and Lemma~\ref{lem:outer} that $R_k(K(t)) \geq \sqrt{\frac{k}{n}}$ for all $t \in [0,2]$.
For the reverse inequality, we first observe for any $A,B \subset \R^n$ that $\op{\conv(A \cup B)}{F} = \conv((\op{A}{F}) \cup (\op{B}{F}))$.
Now, Proposition~\ref{prop:k-rad} together with the equality case in Lemma~\ref{lem:outer} shows that
$\op{P_i}{F} \subset B \coloneqq \sqrt{\frac{k}{n}} \, (\op{\B^n}{F})$ for $i \in \{1,\infty\}$.
Consequently, $\op{K(t)}{F} \subset B$ for all $t \in [0,1]$ and therefore $R_k(K(t)) \leq \sqrt{\frac{k}{n}}$ for all $t \in [0,1]$.
If $n$ is odd or $k \notin \{1,n-1\}$, Proposition~\ref{prop:k-rad} additionally shows $\op{T}{F} \subset B$, which implies that $\op{K(t)}{F} \subset B$ and
$R_k(K(t)) \leq \sqrt{\frac{k}{n}}$ keeps true for all $t \in [1,2]$ in these cases as well.
\end{proof}

\section{Maximal Ellipsoids Contained in Sections}
\label{sec:inner}

The proof of the inequality in Theorem~\ref{thm:inner} reuses large parts of the original proof of Proposition~\ref{prop:ball_thm} from \cite{Ball}.
We therefore first reproduce the main steps of that proof and in the process establish the claimed equality conditions for the original constant upper bound.
Afterwards, we adjust certain parts of the proof to allow us to involve the John asymmetry for the new upper bound and again derive necessary conditions for the equality case.
Once these parts are complete, we turn to the tightness in Theorem~\ref{thm:inner}
and the initially mentioned examples for the range of John asymmetries $s_J(K) \in (1,1+\frac{2}{n})$.

To prove Theorem~\ref{thm:inner}, we require the following two propositions.
The first one gives the maximum of a linear functional evaluated over a general $k$-ellipsoid.
The right-to-left inequality below has essentially been verified in \cite{Ball} by considering the point identified as maximizer in the first case below.
Proving the other direction is a simple application of the Cauchy-Schwarz inequality.
We include it here to help with transparency why precisely these terms appear during the proof of Theorem~\ref{thm:inner}.

\begin{proposition} \label{prop:ellip_lin_opt} 
Let $v^1, \ldots, v^k \in \R^n$ be orthonormal,
$\alpha_1, ..., \alpha_k > 0$, and $b, c \in \R^n$.
For the $k$-ellipsoid 
$E \coloneqq \{ x \in c + \lin(\{v^1, \ldots, v^k\}) :
\sum_{j \in [k]} \frac{((x-c)^T v^j)^2}{\alpha_j^2} \leq 1\}$,
we have
\[
    h_E(b)
    = b^T c + \sqrt{\sum_{j \in [k]} \alpha_j^{2} (b^T v^j)^2}.
\]
The maximum inner product with $b$ is attained by
only $c + (\sum_{j \in [k]} \alpha_j^2 (b^T v^j)^2)^{-\nicefrac{1}{2}} \sum_{j \in [k]} \alpha_j^2 (b^T v^j) v^j$ if $b \notin \lin(\{v^1, \ldots, v^k\})^\perp$, or
every point in $E$, otherwise.
\end{proposition}

The second proposition is taken from \cite[Lemma~5]{Ball}.
It is a key tool in the original proof of Proposition~\ref{prop:ball_thm},
so it is natural that establishing the equality conditions for the constant upper bound in Theorem~\ref{thm:inner} requires an analysis of the equality case in the result below.
For the convenience of the reader, we reproduce the proof of \cite[Lemma~5]{Ball} as first part of the proof below.

\begin{proposition} \label{prop:ball_lemma}
Let $x^1, \ldots, x^m \in \R^n$ satisfy $\sum_{i \in [m]} x^i = 0$ and let $u^1, \ldots, u^m \in \bd(\B^n)$.
Then
\[ 
    \left( \sum_{i \in [m]} (x^i)^T u^i \right)^2 \leq \sum_{i,\ell \in [m]} \gauge{x^i} \gauge{x^\ell} (1 - (u^i)^T u^\ell) =: \gamma,
\]
with equality if and only if for $\xi \coloneqq \sum_{i \in [m]} \gauge{x^i}$ and $u \in \R^n$ such that
$\xi u = \sum_{i \in [m]} \gauge{x^i} u^i$,
there exists $\sigma \in \{-1,1\}$ such that
$x^i = 0$ or $u^i - u = \sigma \frac{\sqrt{\gamma}}{\xi} \frac{x^i}{\gauge{x^i}}$ for all $i \in [m]$.
\end{proposition}
\begin{proof}
W.l.o.g.~we may assume $\xi > 0$.
We define $\lambda_i \coloneqq \frac{\gauge{x^i}}{\xi} \geq 0$ for every $i \in [m]$ such that $\sum_{i \in [m]} \lambda_i = 1$.
Since $\sum_{i \in [m]} x^i = 0$, it follows from the Cauchy-Schwarz inequality and Jensen's inequality that
\begin{align*}
    \left( \sum_{i \in [m]} \frac{(x^i)}{\xi}^T u^i \right)^2
    & = \left\vert \sum_{i \in [m]} \frac{(x^i)}{\xi}^T (u^i - u) \right\vert^2
    \leq \left( \sum_{i \in [m]} \left\vert \frac{(x^i)}{\xi}^T (u^i - u) \right\vert \right)^2
    \leq \left( \sum_{i \in [m]} \lambda_i \gauge{u^i - u} \right)^2
    \\
    & \leq \sum_{i \in [m]} \lambda_i \gauge{u^i - u}^2
    = \sum_{i \in [m]} \lambda_i \gauge{u^i}^2 + \sum_{i \in [m]} \lambda_i \gauge{u}^2 - \sum_{i \in [m]} \lambda_i 2 u^T u^i
    \\
    & = 1 - u^T u
    = \sum_{i,\ell \in [m]} \lambda_i \lambda_\ell (1 - (u^i)^T u^\ell)
    = \sum_{i,\ell \in [m]} \frac{\gauge{x^i} \gauge{x^\ell}}{\xi^2} (1 - (u^i)^T u^\ell).
\end{align*}

Let us now consider the equality case.
By the first line of estimations, there must exist some $\sigma \in \{-1,1\}$ such that
$(x^i)^T (u^i - u) = \sigma \gauge{x^i} \gauge{u^i - u}$ for all $i \in [m]$,
which is fulfilled if and only if $x^i = 0$ or $u^i - u$ is a non-negative multiple of $\sigma x^i$ for all $i \in [m]$.
Additionally, equality in the application of Jensen's inequality holds if and only if there exists some $\rho \geq 0$ with $\gauge{u^i - u} = \rho$ whenever $\lambda_i \neq 0$.
Both equality conditions taken together are equivalent to the claimed equality condition,
up to the precise value of $\rho$.
To this end, we observe that the above computations show in the equality case that
\[
    \rho^2
    = \sum_{i \in [m]} \lambda_i \gauge{u^i - u}^2
    = \sum_{i,\ell \in [m]} \frac{\gauge{x^i} \gauge{x^\ell}}{\xi^2} (1 - (u^i)^T u^\ell)
    = \frac{\gamma}{\xi^2},
\]
so the complete claimed equality condition follows.
\end{proof}

We now turn to the equality case of the inequality in Proposition~\ref{prop:ball_thm},
for which we reproduce the relevant steps from the proof of \cite[Proposition~4]{Ball} at the beginning of the following proof.
The version of the proof provided here involves Proposition~\ref{prop:ellip_lin_opt} more explicitly to clarify some intermediate steps
and streamlines the approach with regard to the changes necessary for the upcoming Lemma~\ref{lem:inner_asym}.

\begin{lemma} \label{lem:Ball_eq}
Let $u^1, \ldots, u^m \in \bd(\B^n)$ and $\lambda_1, \ldots, \lambda_m > 0$ form a John decomposition.
If $E \subset \R^n$ is a $k$-ellipsoid contained in $P \coloneqq \bigcap_{i \in [m]} H_{(u^i,1)}^\leq$ and
\[
    vol_k(E)
    = vol_k \left( \sqrt{ \frac{n (n+1)}{k (k+1)} } \, \B^k \right),
\]
then the following conditions are satisfied:
\begin{enumerate}[(i)]
\item $E$ is a Euclidean $k$-ball of radius $\sqrt{\frac{n (n+1)}{k (k+1)}}$, and
\item the center $c$ of $E$ satisfies $c^T x = \frac{n (n-k)}{k+1}$ for all $x \in \aff(E)$
and $c^T u^i \in \{ 1, \frac{k - n}{k+1} \}$ for all $i \in [m]$.
\end{enumerate}
\end{lemma} 
\begin{proof}
Let $v^1, \ldots, v^k \in \R^n$ be orthonormal vectors, $c \in \R^n$, and $\alpha_1, \ldots, \alpha_k > 0$ such that
\[
    E = \left\{x  \in c + \lin(\{v^1, \ldots, v^k\}) :
        \sum_{j \in [k]} \frac{ ((x-c)^T v^j)^2 }{ \alpha_j^2 } \leq 1 \right\}.
\]
By the inequality between the geometric mean, arithmetic mean, and quadratic mean, we have
\begin{equation} \label{eq:GM_AM_QM_ineq}
    \left( \frac{\vol_k(E)}{\vol_k(\B^k)} \right)^{\nicefrac{1}{k}}
    = \prod_{j \in [k]} \alpha_j^{\nicefrac{1}{k}}
    \leq \sum_{j \in [k]} \frac{\alpha_j}{k}
    \leq \sqrt{ \sum_{j \in [k]} \frac{\alpha_j^2}{k} }.
\end{equation}
For $\kappa \in \{1,2\}$, we can rewrite
\begin{equation} \label{eq:aj_square_sum}
    \sum_{j \in [k]} \alpha_j^\kappa
    = \sum_{j \in [k]} \alpha_j^\kappa (v^j)^T v^j
    = \sum_{j \in [k]} \alpha_j^\kappa (v^j)^T \left( \sum_{i \in [m]} \lambda_i u^i (u^i)^T \right) v^j
    = \sum_{i \in [m]} \lambda_i \sum_{j \in [k]} \alpha_j^\kappa ((u^i)^T v^j)^2,
\end{equation}
which allows us to use the expressions that appear in Proposition~\ref{prop:ellip_lin_opt}.

On the one hand, we upper bound the quadratic mean of the $\alpha_j$ using Proposition~\ref{prop:ellip_lin_opt}.
Since $E \subset P$, we have
\begin{equation} \label{eq:first_QM_ineq}
    c^T u^i + \sqrt{ \sum_{j \in [k]} \alpha_j^2 ((u^i)^T v^j)^2 }
    = h_E(u^i)
    \leq 1
\end{equation}
for all $i \in [m]$.
It follows from \eqref{eq:aj_square_sum} with $\kappa = 2$ that
\begin{align}
\begin{split} \label{eq:inner_cnorm_ineq}
    \sum_{j \in [k]} \alpha_j^2
    & = \sum_{i \in [m]} \lambda_i \sum_{j \in [k]} \alpha_j^2 ((u^i)^T v^j)^2
    \leq \sum_{i \in [m]} \lambda_i (1 - c^T u^i) (1 - c^T u^i)
    \\
    & = \sum_{i \in [m]} \lambda_i - 2 c^T \sum_{i \in [m]} \lambda_i u^i + c^T \left( \sum_{i \in [m]} \lambda_i u^i (u^i)^T \right) c
    = n + \gauge{c}^2.
\end{split}
\end{align}

On the other hand, we upper bound the arithmetic mean of the $\alpha_j$ by applying Proposition~\ref{prop:ball_lemma} to the vectors
$x^i \coloneqq \lambda_i \sum_{j \in [k]} \alpha_j ((u^i)^T v^j) v^j$, $i \in [m]$, and the $u^i$.
We observe that
\[
    \sum_{i \in [m]} x^i
    = \sum_{j \in [k]} \alpha_j \left( \left( \sum_{i \in [m]} \lambda_i u^i \right)^T v^j \right) v^j
    = 0,
\]
as well as $(x^i)^T u^i = \lambda_i \sum_{j \in [k]} \alpha_j ((u^i)^T v^j)^2$ for $i \in [m]$.
We additionally have from \eqref{eq:first_QM_ineq} that
$\gauge{x^i} = \lambda_i \sqrt{ \sum_{j \in [k]} \alpha_j^2 ((u^i)^T v^j)^2 } \leq \lambda_i (1 - c^T u^i)$.
Now, \eqref{eq:aj_square_sum} with $\kappa = 1$ and Proposition~\ref{prop:ball_lemma} show
\begin{align}
\begin{split} \label{eq:aj_sum_square}
    \left( \sum_{j \in [k]} \alpha_j \right)^2
    & = \left( \sum_{i \in [m]} \lambda_i \sum_{j \in [k]} \alpha_j ((u^i)^T v^j)^2 \right)^2
    = \left( \sum_{i \in [m]} (x^i)^T u^i \right)^2
    \\
    & \leq \sum_{i,\ell \in [m]} \gauge{x^i} \gauge{x^\ell} (1 - (u^i)^T u^\ell)
    \leq \sum_{i,\ell \in [m]} \lambda_i \lambda_\ell (1 - c^T u^i) (1 - (u^i)^T u^\ell) (1 - c^T u^\ell).
\end{split}
\end{align}
Expanding the product in the last sum and using the properties of the $u^i$ gives
\begin{equation} \label{eq:inner_general_neq_cnorm_ineq}
    \left( \sum_{j \in [k]} \alpha_j \right)^2
    \leq n^2 - \gauge{c}^2.
\end{equation}

From \eqref{eq:GM_AM_QM_ineq}, \eqref{eq:inner_cnorm_ineq},
and \eqref{eq:inner_general_neq_cnorm_ineq} we obtain that
\begin{equation} \label{eq:inner_general_final_step}
    \left( \frac{\vol_k(E)}{\vol_k(\B^k)} \right)^{\nicefrac{1}{k}}
    \leq \sqrt{ \min \left\{ \frac{n + \gauge{c}^2}{k}, \frac{n^2 - \gauge{c}^2}{k^2} \right\} }
    \leq \sqrt{ \frac{n (n+1)}{k (k+1)} },
\end{equation}
where the right-hand inequality is obtained by considering which term in the minimum is smaller.
This concludes the reproduction of the original proof of Proposition~\ref{prop:ball_thm} from \cite{Ball}.
Since the lemma assumes equality from left to right in \eqref{eq:inner_general_final_step},
we need equality in all intermediate steps above.

First, equality in \eqref{eq:GM_AM_QM_ineq}~and~\eqref{eq:inner_general_final_step} implies
$\alpha_j = \sqrt{ \frac{n (n+1)}{k (k+1)} }$
for all $j \in [k]$, which already proves (i).
Next, equality in \eqref{eq:inner_cnorm_ineq}, and therefore in \eqref{eq:first_QM_ineq} for all $i \in [m]$, gives
$\gauge{x^i} = \lambda_i \sqrt{ \sum_{j \in [k]} \alpha_j^2 ((u^i)^T v^j)^2 } = \lambda_i (1 - c^T u^i)$.
Now, we obtain
\[
    \xi
    \coloneqq \sum_{i \in [m]} \gauge{x^i}
    = \sum_{i \in [m]} \lambda_i (1 - c^T u^i)
    = n
    > 0
\]
as well as
\[
    u
    \coloneqq \sum_{i \in [m]} \frac{\gauge{x^i}}{\xi} u^i
    = \sum_{i \in [m]} \frac{\lambda_i (1 - c^T u^i)}{n} u^i
    = - \frac{1}{n} \sum_{i \in [m]} \lambda_i u^i (u^i)^T c
    = - \frac{1}{n} c.
\]
Equality in \eqref{eq:aj_sum_square} further shows
\[
    \rho
    \coloneqq \frac{\sqrt{ \sum_{i,\ell \in [m]} \gauge{x^i} \gauge{x^\ell} (1 - (u^i)^T u^\ell) }}{\xi}
    = \frac{1}{n} \sum_{j \in [k]} \alpha_j
    = \sqrt{ \frac{k (n+1)}{n (k+1)} }.
\]
By the equality case in Proposition~\ref{prop:ball_lemma},
equality in \eqref{eq:aj_sum_square} gives for each $i \in [m]$ that either $x^i = 0$ and thus $c^T u^i = 1$ by $\gauge{x^i} = \lambda_i (1 - c^T u^i)$,
or $u^i - u = \pm \rho \frac{x^i}{\gauge{x^i}}$ and thus
\[
    \frac{k (n+1)}{n (k+1)}
    = \rho^2
    = \gauge{u^i - u}^2
    = \gauge{u^i}^2 + \gauge{u}^2 - 2 u^T u^i
    = 1 + \frac{\gauge{c}^2}{n^2} + \frac{2}{n} c^T u^i.
\]
Equality in the right-hand inequality in
\eqref{eq:inner_general_final_step} yields
$\gauge{c}^2 = \frac{n (n-k)}{k+1}$.
Thus, the above equality from left to right can be simplified to
\[
    c^T u^i
    = \frac{k (n+1)}{2 (k+1)} - \frac{n}{2} - \frac{n-k}{2 (k+1)} 
    = \frac{k-n}{k+1}.
\]
This verifies the claim about the possible values for $c^T u^i$ in (ii).
Finally, we have for all $j \in [k]$
\[
    c^T v^j
    = \sum_{i \in [m]} \lambda_i (c^T u^i) ((u^i)^T v^j)
    = \frac{k - n}{k+1} \sum_{i \in [m]} \lambda_i (u^i)^T v^j
    = 0,
\]
where we use for each $i \in [m]$ that either
$c^T u^i = \frac{k - n}{k+1}$
or $0 = \gauge{x^i} = \lambda_i \sqrt{ \sum_{j \in [k]} \alpha_j^2 ((u^i)^T v^j)^2 }$
and consequently $u^i \in \lin(\{ v^1, \ldots, v^k \})^\perp$.
It follows $c^T x = \gauge{c}^2 = \frac{n(n-k)}{k+1}$ for all $x \in \aff(E) = c + \lin(\{v^1, \ldots, v^k\})$
and in summary (ii).
\end{proof}

Next, we prove the new upper bound in Theorem~\ref{thm:inner} that depends on the John asymmetry and obtain its equality conditions.
As initially discussed, the proof below reuses large parts of the original proof of Proposition~\ref{prop:ball_thm} in \cite{Ball}.
The main difference is a new inequality that replaces \eqref{eq:inner_general_neq_cnorm_ineq}.
It is again decreasing in $\gauge{c}^2$, but allows us to involve the John asymmetry instead of leading to a constant bound.

\begin{lemma} \label{lem:inner_asym}
Let $u^1, \ldots, u^m \in \bd(\B^n)$ and $\lambda_1, \ldots, \lambda_m > 0$ form a John decomposition.
If $E \subset \R^n$ is a $k$-ellipsoid contained in
$P \coloneqq \bigcap_{i \in [m]} H_{(u^i,1)}^\leq$
such that $E \subset -s P$ for some $s > 0$,
then
\[
    vol_k(E)
    \leq vol_k \left( \sqrt{\frac{n (s+1)}{2 k}} \, \B^k \right).
\]
Moreover, equality implies that the following conditions are satisfied:
\begin{enumerate}[(i)]
\item $E$ is a Euclidean $k$-ball of radius $\sqrt{\frac{n (s+1)}{2 k}}$,
\item the center $c$ of $E$ satisfies $c^T x = \frac{n (s-1)}{2}$ for all $x \in \aff(E)$ and $c^T u^i \in \{1, \frac{1-s}{2}\}$ for all $i \in [m]$, and
\item $s \notin (1,1+\frac{2}{n})$.
\end{enumerate}
\end{lemma}
\begin{proof}
Let $v^1, \ldots, v^k \in \R^n$ be orthonormal vectors, $c \in \R^n$, and $\alpha_1, \ldots, \alpha_k > 0$ such that
\[
    E = \left\{x  \in c + \lin(\{v^1, \ldots, v^k\}) :
        \sum_{j \in [k]} \frac{ ((x-c)^T v^j)^2 }{ \alpha_j^2 } \leq 1 \right\}.
\]
Since $E \subset -s P$,
Proposition~\ref{prop:ellip_lin_opt} implies
\begin{equation} \label{eq:second_QM_ineq}
    - c^T u^i + \sqrt{ \sum_{j \in [k]} \alpha_j^2 ((u^i)^T v^j)^2 }
    = h_E(-u^i)
    \leq s.
\end{equation}
Similar to \eqref{eq:inner_cnorm_ineq}, we now obtain from \eqref{eq:aj_square_sum} with $\kappa = 2$, \eqref{eq:first_QM_ineq}, and \eqref{eq:second_QM_ineq} that
\begin{align}
\begin{split} \label{eq:inner_asym_neg_cnorm_ineq}
    \sum_{j \in [k]} \alpha_j^2
    & = \sum_{i \in [m]} \lambda_i \sum_{j \in [k]} \alpha_j^2 ((u^i)^T v^j)^2
    \leq \sum_{i \in [m]} \lambda_i (1 - c^T u^i) (s + c^T u^i)
    \\
    & = s \sum_{i \in [m]} \lambda_i + (1-s) c^T \sum_{i \in [m]} \lambda_i u^i - c^T \sum_{i \in [m]} \lambda_i u^i (u^i)^T c
    = s n - \gauge{c}^2.
\end{split}
\end{align}
We conclude from \eqref{eq:GM_AM_QM_ineq}, \eqref{eq:first_QM_ineq}, and \eqref{eq:inner_asym_neg_cnorm_ineq} that
\begin{equation} \label{eq:inner_asym_final_step}
    \left( \frac{\vol_k(E)}{\vol_k(\B^k)} \right)^{\nicefrac{1}{k}}
    \leq \sqrt{ \min \left\{ \frac{n + \gauge{c}^2}{k}, \frac{s n - \gauge{c}^2}{k} \right\} }
    \leq \sqrt{ \frac{n (s+1)}{2 k} },
\end{equation}
where the right-hand inequality is obtained by considering which term in the minimum is smaller.
This proves the claimed inequality.

Now, assume that equality is attained from left to right in  \eqref{eq:inner_asym_final_step}.
Then equality must hold in \eqref{eq:GM_AM_QM_ineq} as well,
which implies $\alpha_j = \sqrt{\frac{n (s+1)}{2k}}$ for all $j \in [k]$, i.e., (i).
Additionally, \eqref{eq:inner_cnorm_ineq} and \eqref{eq:inner_asym_neg_cnorm_ineq} must be fulfilled with equality.
As in the proof of Lemma~\ref{lem:Ball_eq},
the former requires
\[ 
    \sqrt{ \sum_{j \in [k]} \alpha_j^2 ((u^i)^T v^j)^2 }
    = 1 - c^T u^i
\]
for all $i \in [m]$.
With this at hand, we see that equality in \eqref{eq:inner_asym_neg_cnorm_ineq}
further implies
\[
    c^T u^i = 1
    \quad \text{or} \quad
    \sqrt{ \sum_{j \in [k]} \alpha_j^2 ((u^i)^T v^j)^2 }
    = s + c^T u^i
\]
for all $i \in [m]$ as well.
It follows that $c^T u^i \in \{1, \frac{1-s}{2}\}$.
If $c^T u^i = 1$, we additionally have
$\sqrt{ \sum_{j \in [k]} \alpha_j^2 ((u^i)^T v^j)^2 } = 1 - c^T u^i = 0$,
which implies $u^i \in \lin(\{v^1, \ldots, v^k\})^\perp$.
Now, we obtain
\[
    c^T v^j
    = \sum_{i \in [m]} \lambda_i (c^T u^i) (u^i)^T v^j
    = \frac{1-s}{2} \sum_{i \in [m]} \lambda_i (u^i)^T v^j
    = 0
\]
for all $j \in [k]$, where we used for the second equality that $c^T u^i = \frac{1-s}{2}$ or $(u^i)^T v^j = 0$ for all $i \in [m]$.
It follows $c^T x = \gauge{c}^2$ for all $x \in \aff(E) = c + \lin(\{v^1, \ldots, v^k\})$.
Since equality in the right-hand inequality in \eqref{eq:inner_asym_final_step} requires $\gauge{c}^2 = \frac{n (s-1)}{2}$, we completed the proof of (ii).

Finally, we turn to (iii),
for which we assume $s < 1+\frac{2}{n}$.
Now, (ii) shows $\gauge{c}^2 = \frac{n (s-1)}{2} < 1$.
Consequently, the Cauchy-Schwarz inequality yields $c^T u^i < 1$,
which implies by (ii) that $c^T u^i = \frac{1-s}{2}$, for all $i \in [m]$.
If $c \neq 0$, this would mean that all $u^i$ lie in a common hyperplane,
contradicting that they belong to a John decomposition.
This implies $c = 0$ and therefore $0 = \gauge{c}^2 = \frac{n (s-1)}{2}$, which means $s = 1$.
\end{proof}

With the above lemma, the inequality and equality conditions in Theorem~\ref{thm:inner} are established.
We turn to its tightness in the following.

We should notice the following concerning the upper bound in Theorem~\ref{thm:inner}:
Since $\frac{s_J(K)+1}{2} = \frac{n+1}{k+1}$ if and only if $s_J(K) = 2 \frac{n+1}{k+1} - 1 =: s_{n,k}$,
it is natural that the examples proving tightness for the theorem require different treatment depending on whether $s_J(K)$ is above or below the threshold $s_{n,k}$.
We begin with the case $s_J(K) \geq s_{n,k}$.

\begin{example} \label{ex:inner_high_asym}
Let $k \in [n-1]$, $T \in \CK^n$ a simplex with $\jel(T) = \B^n$,
and $E \subset T$ the inscribed Euclidean $k$-ball of radius $\rho = \sqrt{ \frac{n (n+1)}{k (k+1)} }$
of a $k$-face of $T$ (which exists by Proposition~\ref{prop:ball_thm}).
Then for any $s \in [s_{n,k},n]$, the body
\[
    K
    \coloneqq T \cap (-s T)
\]
satisfies $\jel(K) = \B^n$, $s_J(K) = s$, and $E \subset K$.
\end{example}
\begin{proof}
We begin by proving $E \subset - s_{n,k} T \subset -s T$, where the latter inclusion is clear from
$0 \in T$ and $s_{n,k} \leq s$.
Once the first inclusion is also established, we obtain $E \subset K$.
Since $T$ is in John position,
its vertices $x^1, \ldots, x^{n+1}$ belong to $\bd(n \B^n)$
and satisfy
\[
    T
    = \bigcap_{i \in [n+1]} H_{(-x^i,n)}^\leq.
\]
We assume they are indexed such that $F \coloneqq \conv(\{x^1, \ldots, x^{k+1}\})$
is the $k$-face of $T$ containing $E$.
Then with $V \coloneqq \aff(F) - x^{k+1}$ the linear $k$-space parallel to $\aff(F)$,
$B \coloneqq \B^n \cap V$,
and $c \coloneqq \sum_{j \in [k+1]} \frac{x^j}{k+1}$,
we have $E = c + \rho B$.
Hence, we may establish $E \subset -s_{n,k} T$ by proving
\[
    c^T x^i + \rho h_B(x^i)
    \leq n s_{n,k}
\]
for all $i \in [n+1]$.
Since $h_B(x^i) = \gauge{\op{x^i}{V}}$,
our next goal is to compute $\op{x^i}{V}$.
The first step to this end is recognizing that $\{x^j - x^{k+1} : j \in [k]\}$ is a basis of $V$.
Since for any $i,\ell \in [n+1]$, $i \neq \ell$, we have $(x^i)^T x^\ell = -n$,
it follows for $i > k+1$ and $j \in [k]$ that $(x^i)^T (x^j - x^{k+1}) = -n - (-n) = 0$.
Thus, $x^i \in V^\perp$ and $\op{x^i}{V} = 0$, so
\[
    c^T x^i + \rho h_B(x^i)
    = \sum_{j \in [k+1]} \frac{(x^j)^T x^i}{k+1} + \rho \gauge{\op{x^i}{V}}
    = \sum_{j \in [k+1]} \frac{-n}{k+1}
    = -n
    \leq n s_{n,k}
\]
for $i > k+1$.
If instead $i \in [k+1]$, we first observe that
$c = \sum_{j \in [k+1]} \frac{x^j}{k+1} = \sum_{j\in[n+1]\setminus[k+1]} \frac{-x^j}{k+1} \in V^\perp$
and $x^i - c \in V$,
so $x^i = c + (x^i - c)$ implies $\op{x^i}{V} = x^i - c$.
Consequently,
\[
    \gauge{\op{x^i}{V}}^2
    = \gauge{x^i}^2 + \gauge{c}^2 - 2 c^T x^i
    = n^2 + \frac{n (n-k)}{k+1} - 2 \frac{n (n-k)}{k+1}
    = \frac{n (n+1) k}{k+1},
\]
where $\gauge{c}^2 = \frac{n (n-k)}{k+1}$ and
$c^T x^i = \frac{n (n-k)}{k+1}$
can be computed directly or taken from Lemma~\ref{lem:Ball_eq}.
Now,
\[
    c^T x^i + \rho h_B(x^i)
    = \frac{n (n-k)}{k+1} + \sqrt{ \frac{n (n+1)}{k (k+1)} } \, \sqrt{ \frac{n (n+1) k}{k+1} }
    = \frac{n (n-k)}{k+1} + \frac{n (n+1)}{k+1}
    = n s_{n,k}
\]
finishes the proof of $E \subset - s_{n,k} T$.

Next, we observe from the definition of $K$ and $s_{n,k} > 1$
that $\B^n \subset K \subset T$,
so $\jel(K) = \B^n$.

Lastly, we have by $s \leq n$, $\jel(T) = \B^n$, and $s_J(T) = n$ that
\[
    \frac{s}{n} T
    \subset T \cap (-s T)
    = K
    \subset (s^2 T) \cap (-s T)
    = -s K
    \subset -s T
\]
and thus $s_J(K) = s$.
\end{proof}

To verify the upcoming Examples~\ref{ex:inner_mid_asym}~and~\ref{ex:inner_small_asym},
as well as Example~\ref{ex:john_rounding} in the appendix,
we prove the following lemma.
It provides a particular family of polytopes in John position.
Recall for arbitrary sets $A,B$ that $B^A$ denotes the set of all functions $f \colon A \to B$.

\begingroup
\newcommand\num{{|J|}}
\begin{lemma} \label{lem:construction}
Let $v^1, \ldots, v^n \in \R^n$ be an orthonormal basis,
$J \subset [n-1]$,
and $\tau \in \left[\sqrt{\frac{n - \num}{(\num+1) n}},1\right]$.
For $\delta \in \{-1, 1\}^J$ and $\sigma \in \{-1, 1\}^{[n-1]}$, define the Euclidean unit vectors
\begin{align*}
    a^{\delta,\tau} & \coloneqq \sum_{j \in J} \delta(j) \sqrt{\frac{1-\tau^2}{\num}} \, v^j +  \tau v^n,
        \\  
    b^{\sigma,\tau} & \coloneqq \sum_{j \in J} \sigma(j) \sqrt{\frac{(\num+1) n \tau^2 +\num -n}{\num n^2 \tau^2}} \, v^j + \sum_{j \in [n-1] \setminus J} \sigma(j) \frac{\sqrt{n \tau^2 + 1}}{n \tau} \, v^j -
    \frac{1}{n \tau} v^n,
\intertext{and define the polyhedron}
    P(J,\tau) & \coloneqq \bigcap_{\delta \in \{-1,1\}^J} H_{(a^{\delta,\tau},1)}^\leq \cap \bigcap_{\sigma \in \{-1,1\}^{[n-1]}} H_{(b^{\sigma,\tau},1)}^\leq.
\end{align*}
Then
any $K \in \CK^n$ with $\B^n \subset K \subset P(J,\tau)$ satisfies $\jel(K) = \B^n$.
\end{lemma}
\begin{proof}
It is straightforward to verify that the $a^{\delta,\tau}$ and $b^{\sigma,\tau}$ are well-defined Euclidean unit vectors.
By Proposition~\ref{prop:john}, it thus suffices to verify that $\{a^{\delta,\tau}, b^{\sigma,\tau} : \delta \in \{-1,1\}^J, \sigma \in \{-1,1\}^{[n-1]} \}$ is the set of vectors of some John decomposition.
We define
\[
    \lambda_a
    \coloneqq \frac{1}{2^\num} \frac{n}{n \tau^2 + 1}
    > 0
        \quad \text{and} \quad
    \lambda_b
    \coloneqq \frac{1}{2^{n-1}} \frac{n^2 \tau^2}{n \tau^2 +1}
    > 0,
\]
which satisfy
\begin{equation}
\label{eq:constr_lem_vector}
    \sum_{\delta \in \{-1,1\}^J} \lambda_a a^{\delta,\tau} + \sum_{\sigma \in \{-1,1\}^{[n-1]}} \lambda_b b^{\sigma,\tau}
    = 2^\num \lambda_a \tau v^n - 2^{n-1} \lambda_b \frac{1}{n \tau} v^n = 0.
\end{equation}
Moreover,
\[
    (v^i)^T \left( \sum_{\delta \in \{-1,1\}^J} \lambda_a a^{\delta,\tau} (a^{\delta,\tau})^T \right) v^j
    = \begin{cases}
        0, & \text{if } i \neq j,
        \\
        0, & \text{if } i = j \notin J \cup \{ n \},
        \\
        2^{l} \lambda_a \frac{1 - \tau^2}{\num}, & \text{if } i = j \in J,
        \\
        2^{l} \lambda_a \tau^2, & \text{if } i = j = n,
    \end{cases}
\]
and
\[
    (v^i)^T \left( \sum_{\sigma \in \{-1,1\}^{[n-1]}} \lambda_b b^{\sigma,\tau} (b^{\sigma,\tau})^{T} \right) v^j
    = \begin{cases}
        0, & \text{if } i \neq j,
        \\
        2^{n-1} \lambda_b \frac{n \tau^2 + 1}{n^2 \tau^2}, & \text{if } i = j \notin J \cup \{ n \},
        \\
        2^{n-1} \lambda_b \frac{(\num+1) n \tau^2 + \num - n}{\num n^2 \tau^2}, & \text{if } i = j \in J,
        \\
        2^{n-1} \lambda_b \frac{1}{n^2 \tau^2}, & \text{if } i = j = n.
    \end{cases}
\]
for all $i, j \in [n]$.
If $J \neq \emptyset$, then
\[
    2^\num \lambda_a \frac{1 - \tau^2}{\num} + 2^{n-1} \lambda_b \frac{(\num+1) n \tau^2 + \num - n}{\num n^2 \tau^2}
    = \frac{n (1-\tau^2) + (\num+1) n \tau^2 + \num - n}{\num (n \tau^2 + 1)}
    = 1,
\]
so we anyway obtain
\[
    (v^i)^T \left( \sum_{\delta \in \{-1,1\}^J} \lambda_a a^{\delta,\tau} (a^{\delta,\tau})^T + \sum_{\sigma \in \{-1,1\}^{[n-1]}} \lambda_b b^{\sigma,\tau} (b^{\sigma,\tau})^T \right) v^j
    = \begin{cases}
        0, & \text{if } i \neq j,
        \\
        1, & \text{if } i = j
    \end{cases}
\]
for all $i, j \in [n]$.
Since $v^1, \ldots, v^n$ is an orthonormal basis of $\R^n$,
we conclude that
\begin{equation}
\label{eq:constr_lem_matrix}
    \sum_{\delta \in \{-1,1\}^J} \lambda_a a^{\delta,\tau} (a^{\delta,\tau})^T + \sum_{\sigma \in \{-1,1\}^{[n-1]}} \lambda_b b^{\sigma,\tau} (b^{\sigma,\tau})^T
    = I_n.
\end{equation}
With (\ref{eq:constr_lem_vector}) and (\ref{eq:constr_lem_matrix}) established,
the proof is complete as discussed above.
\end{proof}
\endgroup

Now, we turn to the tightness in Theorem~\ref{thm:inner}
for $s_J(K) \leq s_{n,k}$.
As already observed in Lemma~\ref{lem:inner_asym},
tightness anyway requires $s_J(K) \geq 1 + \frac{2}{n}$ (besides $s_J(K)=1$, which is already covered by Corollary~\ref{cor:ball_sym_improv}).
Let us point out that $k \leq n-1$ always gives $s_{n,k} \geq 1 + \frac{2}{n}$.

\begin{example} \label{ex:inner_mid_asym}
Let $v^1, \ldots, v^n \in \R^n$ be an orthonormal basis,
$k \in [n-1]$,
$s \in [1+\frac{2}{n}, s_{n,k}]$,
and $\rho = \sqrt{\frac{n (s+1)}{2}}$.
Define $c \coloneqq \sqrt{\frac{n (s-1)}{2}} \, v^n$ and
\[
    K
    \coloneqq \conv \left(
        \B^n \cup
        \left\{
            c \pm \rho v^j, -\frac{1}{s} c \pm \frac{\rho}{s} v^j:
            j \in [k]
        \right\}
    \right).
\]
Then $K$ satisfies $\jel(K) = \B^n$,
$s(K) < s_J(K) = s$,
and $E \coloneqq c + \frac{\rho}{\sqrt{k}} (\B^n \cap \lin(\{ v^1, \ldots, v^k \})) \subset K$.
\end{example}

Let us point out that the strict inequality $s(K) < s_J(K)$
shows that the John asymmetry in Theorem~\ref{thm:inner}
cannot be replaced with the Minkowski asymmetry without changing the formula.

\begin{proof}
To prove $\jel(K) = \B^n$,
we shall use Lemma~\ref{lem:construction}
and show $K \subset P([n-1]\setminus[k],\tau)$
for $\tau \coloneqq \sqrt{\frac{2}{n (s-1)}}$.
We first observe that the choice of $s$ guarantees
\[
    \sqrt{\frac{n - (n - (k+1))}{n((n - (k+1))+1)}}
    = \sqrt{\frac{k+1}{n (n-k)}}
    = \sqrt{\frac{2}{n (s_{n,k} - 1)}}
    \leq \tau
    \leq \sqrt{\frac{2}{n (1+\frac{2}{n} - 1)}}
    = 1,
\]
as required when applying Lemma~\ref{lem:construction} for the index set $J = [n-1] \setminus [k]$.

To verify $K \subset P([n-1] \setminus [k],\tau)$,
we reuse the notation of Lemma~\ref{lem:construction} and let $\delta \in \{-1,1\}^{[n-1] \setminus [k]}$ and $\sigma \in \{-1,1\}^{[n-1]}$.
We need to check for all $x \in K$ that $x^T a^{\delta,\tau}, x^T b^{\sigma,\tau} \leq 1$,
where we may disregard all $x \in \B^n$ by the Cauchy-Schwarz inequality.
For any $j \in [k]$, we have
\[
    (c \pm \rho v^j)^T a^{\delta,\tau}
    = \sqrt{ \frac{n (s-1)}{2} } \, \tau
    = \sqrt{ \frac{n (s-1)}{2} } \, \sqrt{ \frac{2}{n (s-1)} }
    = 1
\]
and
\[
    (c \pm \rho v^j)^T b^{\sigma,\tau}
    = \pm \sigma(j) \rho \frac{\sqrt{n \tau^2 + 1}}{n \tau} - \sqrt{\frac{n (s-1)}{2}} \frac{1}{n \tau}
    = \pm \sigma(j) \frac{s+1}{2} - \frac{s-1}{2}
    \in \{-s,1\}.
\]
This also yields $-\frac{1}{s} (c \pm \rho v^j)^T a^{\delta,\tau},
-\frac{1}{s} (c \pm \rho v^j)^T b^{\sigma,\tau} \leq 1$.
Hence, we indeed have $K \subset P([n-1]\setminus[k],\tau)$.

For $E \subset K$,
we observe that $K$ contains the regular $k$-cross-polytope
$C \coloneqq \conv(\{ c \pm \rho v^j : j \in [k] \})$.
It is straightforward to see that the Euclidean circumball of $C$ within $\aff(C)$ is a ball of radius $\rho$ centered at $c$.
It is further well-known that the Euclidean circum- and inball of $C$ within $\aff(C)$ are concentric and have radius ratio $\sqrt{k}$.
We conclude that $E$ is the Euclidean inball of $C$ within $\aff(C)$ and therefore $E \subset C \subset K$.

Lastly, we turn to the claim $s(K) < s_J(K) = s$.
It is clear that $-K \subset sK$, so $s_J(K) \leq s$ by $\jel(K) = \B^n$.
Moreover, since $K$ contains $E$,
a Euclidean $k$-ball of radius $\sqrt{\frac{n (s+1)}{2k}}$,
Lemma~\ref{lem:inner_asym} implies $s_J(K) \geq s$ and in summary $s_J(K) = s$.

Finally, we show for $\varepsilon > 0$ small enough that
$-\frac{1}{s} K + \varepsilon v^n \subset \inte(K)$, proving $s(K) < s_J(K)$.
If $\varepsilon < 1 - \frac{1}{s}$, then clearly
\[
    - \frac{1}{s} \B^n + \varepsilon v^n
    \subset \left( \frac{1}{s} + \varepsilon \right) \B^n
    \subset \inte(\B^n)
    \subset \inte(K).
\]
If additionally
$\varepsilon < \frac{2}{s} \sqrt{\frac{n (s-1)}{2}}$,
we have by $\frac{1}{s} (c \pm \rho v^j) \in (0,c \pm \rho v^j) \subset \inte(K)$ also
\[
    -\frac{1}{s} (c \pm \rho v^j) + \varepsilon v^n
    \in \left( -\frac{1}{s} (c \pm \rho v^j), \frac{1}{s} (c \mp \rho v^j) \right)
    \subset \inte(K).
\]
If further
$\varepsilon < (1 - \frac{1}{s^2}) \sqrt{ \frac{n (s-1)}{2} }$,
we have by
$\frac{1}{s^2} (c \pm \rho v^j) \in (0,c \pm \rho v^j) \subset \inte(K)$
together with
$c \pm \frac{\rho}{s^2} v^j \in [c-\rho v^j,c+\rho v^j] \subset K$
that
\[
    \frac{1}{s^2} (c \pm \rho v^j) + \varepsilon v^n
    \in \left( \frac{1}{s^2} (c \pm \rho v^j), c \pm \frac{\rho}{s^2} v^j \right)
    \subset \inte(K).
\]
Hence, if all three conditions on $\varepsilon$ are satisfied,
then all points in the generating set for $-\frac{1}{s} K + \varepsilon v^n$ are in the interior of $K$.
It follows $-\frac{1}{s} K + \varepsilon v^n \subset \inte(K)$ as claimed.
\end{proof}

Our final goal for this section is to provide an example
that shows which inner $k$-radii are possible for a convex body $K \in \CK^n$ in John position with $s_J(K) \in (1,1+\frac{2}{n})$.
Unfortunately, we do not know if the example we provide
gives the extremal value for all combinations of
$n$, $k$, and $s_J(K)$.
We can only prove this for the planar case,
which we do in the following section.

We expect an extremal $k$-ellipsoid for the John asymmetry range $s_J(K) \in (1,1+\frac{2}{n})$ to satisfy conditions analogous to those obtained in Lemmas~\ref{lem:Ball_eq}~and~\ref{lem:inner_asym}.
Qualitatively, they should remain the same,
which means that the $k$-ellipsoid should again be a $k$-ball of appropriate radius such that the line connecting the origin and the center $c$ of the $k$-ball is perpendicular to the affine hull of the $k$-ball.
We also expect that $c$ again has to have a specific Euclidean norm and can have only two specific values as inner product with the vectors in the John decomposition.
The proof of Theorem~\ref{thm:planar_diam} below shows that these assumptions are true at least in the planar case
(see Remark~\ref{rem:planar_diam} for details).
The upcoming Example~\ref{ex:inner_small_asym} is constructed with these assumptions on the extremal $k$-ellipsoids in mind.

The statement and proof of the example require
the following formulas and some analytic properties of these terms.
For given $n$ and $k \in [n-1]$,
we define for all $s \in [1,1+\frac{2}{n}]$
\begin{align*}
    \zeta_{n,k}(s)
    & \coloneqq \sqrt{ \left( 2 \left( 1+\frac{2}{n}-s \right) - (s^2-1) \right)^2
    + 8 \left( 1+\frac{2}{n}-s \right) (s^2-1) \frac{n-k}{n} }
\intertext{
and
}
    \mu_{n,k}(s)
    & \coloneqq \frac{n}{8 (k+1)} \left( n (s-1)^2 + 4k(s+1) + 4 + n \zeta_{n,k}(s) \right).
\intertext{
For $\mu \in [n,n+1]$, we further define
}
    \tau_{n,k}(\mu)
    & \coloneqq \frac{k \sqrt{\mu-n} + \sqrt{(k (\mu-n)+\mu-k) \mu}}{k (\mu-n) + \mu}.
\end{align*}
The purpose of $\zeta_{n,k}(s)$ is to simplify the statement of $\mu_{n,k}(s)$.
The latter will be used to define the bodies in Example~\ref{ex:inner_small_asym}.
Lastly, $\tau_{n,k}(\mu_{n,k}(s))$ will take the role of $\tau$ in an application of Lemma~\ref{lem:construction}.
The lemma below collects the required analytical properties of these functions.
Since the proof is purely technical,
we postpone it to Appendix~\ref{app:technical}.

\begin{lemma} \label{lem:mu_tau_prop}
Let $k \in [n-1]$. Then
\begin{enumerate}[(i)]
\item $\mu_{n,k} \colon [1,1+\frac{2}{n}] \to \R$ is strictly increasing with
$\mu_{n,k}(1) = n$ and $\mu_{n,k}(1+\frac{2}{n}) = n+1$,
\item $\tau_{n,k} \colon [n,n+1] \to \R$ is strictly increasing with $\tau_{n,k}(n) = \sqrt{\frac{n-k}{n}}$ and $\tau_{n,k}(n+1) = 1$,
\item for all $\mu \in [n,n+1]$,
the in $t$ quadratic equation
$(t \sqrt{\mu-n} - 1)^2 = \mu \frac{1-t^2}{k}$
has a unique non-negative solution, namely $t = \tau_{n,k}(\mu)$, and
\item for all $s \in [1,1+\frac{2}{n}]$, we have
$(s-1) \tau_{n,k}(\mu_{n,k}(s)) = \frac{2}{n} \sqrt{\mu_{n,k}(s)-n}$.
\end{enumerate}
\end{lemma}

With the above properties, the proof of the following example
is rather straightforward.

\begin{example} \label{ex:inner_small_asym}
Let $v^1, \ldots, v^n \in \R^n$ be an orthonormal basis,
$k \in [n-1]$,
$s \in [1, 1 + \frac{2}{n}]$,
and $\rho = \sqrt{\mu_{n,k}(s)}$.
Define $c \coloneqq \sqrt{\rho^2-n} \, v^n$ and
\[
    K
    \coloneqq \conv \left(
        \B^n \cup
        \left\{
            c \pm \rho v^j, -\frac{1}{s} c \pm \frac{\rho}{s} v^j:
            j \in [k]
        \right\}
    \right).
\]
Then $K$ satisfies $\jel(K) = \B^n$,
$s_J(K) = s$,
and $E \coloneqq c + \frac{\rho}{\sqrt{k}} (\B^n \cap \lin(\{ v^1, \ldots, v^k \})) \subset K$.
\end{example}

Note that Lemma~\ref{lem:mu_tau_prop}~(i) shows that the $k$-ellipsoid $E$ above has radius at least $\sqrt{\frac{n}{k}}$ as mentioned in the introduction.

\begin{proof}
Lemma~\ref{lem:mu_tau_prop}~(i) shows
$\rho \in [\sqrt{n},\sqrt{n+1}]$,
so $c$ is well-defined.
The inclusion $E \subset K$ is shown in the same way as for
Example~\ref{ex:inner_mid_asym}.

Next, we turn to the claim $\jel(K) = \B^n$.
We again want to use Lemma~\ref{lem:construction}
and show for $\tau \coloneqq \tau_{n,k}(\rho^2)$ that
$K \subset P([k], \tau)$.
Lemma~\ref{lem:mu_tau_prop}~(i)~and~(ii) hereby show that $\tau$ is well-defined and satisfies
$\sqrt{\frac{n-k}{(k+1)n}} \leq \tau \leq 1$.

Now, let $\delta \in \{-1, 1\}^{[k]}$ and
$\sigma \in \{-1, 1\}^{[n-1]}$ be chosen arbitrarily.
Using the notation in Lemma~\ref{lem:construction},
we need to prove for all $j \in [k]$ that
$(c \pm \rho v^j)^T a^{\delta,\tau}, (c \pm \rho v^j)^T b^{\sigma,\tau} \in [-s, 1]$.

To prove
$(c \pm \rho v^j)^T a^{\delta,\tau} \in [-s,1]$,
we show
\[
    \rho \sqrt{\frac{1-\tau^2}{k}}
    = 1 - \tau \sqrt{\rho^2-n}
    \leq s + \tau \sqrt{\rho^2-n}.
\]
The right-hand inequality is clear.
Since $\tau \leq 1$ and $\sqrt{\rho^2-n} \leq 1$,
the left-hand equality is equivalent to
\[
    \rho^2 \frac{1-\tau^2}{k}
    = \left( 1 - \tau \sqrt{\rho^2-n} \right)^2.
\]
This is true by Lemma~\ref{lem:mu_tau_prop}~(iii).

Second, to prove
$(c \pm \rho v^j)^T b^{\sigma,\tau} \in [-s,1]$,
we show
\begin{equation} \label{eq:inner_small_asym_second}
    \rho \sqrt{\frac{(k+1) n \tau^2 + k - n}{k n^2 \tau^2}}
    = 1 + \frac{\sqrt{\rho^{2} - n}}{n \tau}
    = s - \frac{\sqrt{\rho^{2} - n}}{n \tau}.
\end{equation}
The left-hand equality is equivalent to
\[
    \rho^2 \frac{(k+1) n \tau^2 + k - n}{k}
    = \left( n \tau + \sqrt{\rho^2-n} \right)^2,
\]
where it is straightforward to verify that
\[
    \left( n \tau + \sqrt{\rho^2-n} \right)^2 - \rho^2 \frac{(k+1) n \tau^2 + k - n}{k}
    = n \left( \rho^2 \frac{1-\tau^2}{k} - \left( \tau \sqrt{\rho^2-n} - 1 \right)^2 \right).
\]
The right-hand term equals $0$ by Lemma~\ref{lem:mu_tau_prop}~(iii).
Lastly, the right-hand equality in \eqref{eq:inner_small_asym_second} is equivalent to
\[
    s-1
    = 2 \frac{\sqrt{\rho^2-n}}{n \tau},
\]
which is true by Lemma~\ref{lem:mu_tau_prop}~(iv).

Finally, for $s_J(K) = s$,
we again observe that clearly $K \subset -sK$,
so $s_J(K) \leq s$.
Moreover, we have seen above
for $b^{\sigma,\tau}$ with $\sigma(1) = -1$ that $h_K(b^{\sigma,\tau}) = 1$ and $(c+\rho v^1)^T b^{\sigma,\tau} = -s$.
Thus, $c+\rho v^1 \in K \subset -s_J(K) K \subset H_{(-b^{\sigma,\tau},s_J(K))}^\leq$
implies $s_J(K) \geq s$ and altogether $s_J(K) = s$.
\end{proof}

\section{The Planar Case}
\label{sec:planar_diam}

This section is devoted to the proof of Theorem~\ref{thm:planar_diam}.
The theorem proves that Example~\ref{ex:inner_small_asym} provides the extreme case at least for $n=2$ and $k=1$.
To execute the proof, we require some auxiliary results.

The first two auxiliary results collect some direct consequences of Proposition~\ref{prop:john}.
The probably most well-known such consequence is John's result on the approximation of convex bodies by $n$-ellipsoids \cite{Jo}.
Here, we require a sharpening of this result from \cite[Theorem~7.1]{BrKoSharp}
(see also \cite[Theorem~9]{BeFr} for the analogous result concerning Loewner ellipsoids).
A detailed discussion about the tightness and the differences in the lower bound compared to \cite{BrKoSharp} is provided in Appendix~\ref{app:sJ}.

\begin{proposition} \label{prop:john_rounding}
Let $K \in \CK^n$ with $\jel(K) = \B^n$. Then the smallest $\rho \geq 0$ with $K \subset \rho \B^n$ satisfies
\[
    s_J(K) \leq \rho \leq \sqrt{n s_J(K)}.
\]
Both bounds are best possible for any prescribed $s_J(K) \in [1,n]$.
\end{proposition}

The next proposition collects properties that have essentially (up to their equality cases) been obtained in \cite[Remarks]{Ball}.
Tracking the equality cases is needed to obtain the characterization of the equality case in Theorem~\ref{thm:planar_diam}.
The equality case in (iii) is needed for the proof of Theorem~\ref{thm:aff_d_r_ratio} in the next section.
For the convenience of the reader, we provide a short proof of the full lemma.

\begin{proposition} \label{prop:john_vector_prop}
Let $K \in \CK^n$ with $\jel(K) = \B^n$. Then any $x,y \in K$ satisfy
\begin{enumerate}[(i)]
\item $x^T y \geq -n$, with equality if and only if every vector $u \in \bd(\B^n)$ in a John decomposition for $K$ satisfies $x^T u = 1$ or $y^T u = 1$,
\item $\gauge{x-y} \leq \sqrt{2 (\max\{ \gauge{x}^2, \gauge{y}^2 \} + n)}$, with equality if and only if $\gauge{x} = \gauge{y}$ and $x^T y = -n$, and
\item $\gauge{x-y} \leq \sqrt{2 n (n+1)}$, with equality if and only if $\gauge{x} = \gauge{y} = n$ and $x \neq y$.
\end{enumerate}
\end{proposition}
\begin{proof}
Let $u^1, \ldots, u^m \in \bd(K) \cap \bd(\B^n)$ and weights $\lambda_1, \ldots, \lambda_m > 0$ form a John decomposition. Then
\[
    0
    \leq \sum_{i \in [m]} \lambda_i (1 - x^T u^i) (1 - y^T u^i)
    = \sum_{i \in [m]} \lambda_i - (x + y)^T \sum_{i \in [m]} \lambda_i u^i + x^T \sum_{i \in [m]} \lambda_i u^i (u^i)^T y
    = n + x^T y.
\]
This already yields (i).
For (ii), use (i) to obtain
\[
    \gauge{x-y}^2
    = \gauge{x}^2 + \gauge{y}^2 - 2 x^T y
    \leq 2 \max \{ \gauge{x}^2, \gauge{y}^2 \} + 2 n
\]
and obviously also the claimed equality case.

Lastly, for the inequality in (iii) we use (ii) and that $K \subset n \B^n$
by Proposition~\ref{prop:john_rounding}.
In the equality case, (ii) already yields
$\gauge{x} = \gauge{y} = n$ and $x \neq y$.
Finally, assume $\gauge{x} = \gauge{y} = n$ and $x \neq y$.
Since $- \frac{x}{n} \in \bd(\B^n) \subset K$,
(i) shows that any $u^i$ satisfies
$x^T u^i = 1$ or $(- \frac{x}{n})^T u^i = 1$.
In the second case, the Cauchy-Schwarz inequality and its equality case show that $u^i = - \frac{x}{n}$.
This must apply for at least one $u^i$, say $u^1$,
as otherwise all $u^i$ would lie in a common hyperplane,
which is impossible for the set of vectors from a John decomposition.
The same arguments apply for $y$ instead of $x$.
Since $x \neq y$ gives
$u^1 = -\frac{x}{n} \neq -\frac{y}{n}$,
it thus follows $y^T u^1 = 1$,
i.e., $x^T y = -n$.
Together with $\gauge{x} = \gauge{y} = n$,
(ii) now shows that equality holds in (iii).
\end{proof}

\begin{remark}
Similar to the observation made in \cite[Remarks]{Ball},
let us point out that Proposition~\ref{prop:john_vector_prop}~(ii) together with Proposition~\ref{prop:john_rounding} provides an easy alternative proof of Theorem~\ref{thm:inner} for $k=1$.
The equality conditions obtained from Proposition~\ref{prop:john_vector_prop}~(ii) for the resulting inequality $\gauge{x-y} \leq \sqrt{2 n (s_J(K)+1)}$ are
\[
    \gauge{x} = \gauge{y} = \sqrt{n s_J(K)}
        \quad \text{and} \quad
    x^T y = -n.
\]
Using $\gauge{x+y}^2 = \gauge{x}^2 + \gauge{y}^2 + 2 x^T y$,
we may equivalently replace these conditions by
\[
    \gauge{x} = \gauge{y} = \sqrt{n s_J(K)}
        \quad \text{and} \quad
    \gauge{\frac{x+y}{2}} = \sqrt{ \frac{n (s_J(K) - 1)}{2}}.
\]
Written in this way, the equality conditions are analogous to the conditions in Theorem~\ref{thm:planar_diam} and can easily be seen to really be equivalent to the conditions in Theorem~\ref{thm:inner}.

Also note that the equality conditions cannot be reduced to $\gauge{x} = \gauge{y} = \sqrt{n s_J(K)}$
and $x \neq y$ for general values of $s_J(K) \in [1,n]$ as Proposition~\ref{prop:john_vector_prop}~(iii) might suggest.
For example, any vertex of the cube $[-1,1]^n$ has Euclidean norm $\sqrt{n}$,
but only antipodal pairs of its vertices are at distance $2 \sqrt{n}$ from each other.
\end{remark}

The last auxiliary result needed takes care of some technical computations.
Recall the definition of $D_s$ in \eqref{eq:Ds}.

\begin{lemma} \label{lem:technical}
Let $s \in (1,2)$ and define
\[
    \xi_s^*
    \coloneqq \sqrt{\frac{D_s^2}{2} - 2}
    = \sqrt{ \frac{s^2+1 + \sqrt{4 (2-s)^2 + (s^2-1)^2}}{2} }
    \in (s,\sqrt{2s}).
\]
Then the function
\[
    f_s \colon (s,\infty) \to \R,
    f_s(\xi) = \xi^2 + \frac{(2-s)^2}{\xi^2-s^2}
\]
attains its unique maximum on $[\xi_s^*,\sqrt{2 s}]$
at $\xi_s^*$ with $f_s(\xi_s^*) = D_s^2 - 5 = s^2 + \sqrt{4 (2-s)^2 + (s-1)^2}$.
\end{lemma}
\begin{proof}
We start with $s < \xi_s^* < \sqrt{2s}$.
For the left-hand inequality,
we just observe that $s \in (1,2)$ yields $\sqrt{4 (2 - s)^2 + (s^2-1)^2} > s^2-1$.
Next, $\xi_s^* < \sqrt{2s}$ is equivalent to
\[
    \sqrt{4 (2 - s)^2 + (s^2-1)^2}
    < 2s - (s-1)^2.
\]
Since the right-hand side is positive,
we may square both sides and simplify to
\[
    0 < 8 (2 - s) (s - 1) (s + 1),
\]
which is true by $s \in (1,2)$.

Now, $f_s$ is clearly differentiable in $\xi$ with derivative
\[
    f_s'(\xi)
    = 2\xi - \frac{2\xi (2-s)^2}{(\xi^2-s^2)^2}
    = \frac{ 2\xi \left( (\xi^2-s^2)^2 - (2-s)^2 \right)}{(\xi^2-s^2)^2},
\]
which becomes zero precisely when $\xi=0$ or $\xi^2 = s^2 \pm (2 - s)$.
The only such $\xi$ also satisfying $\xi > s$ is $\xi = \sqrt{s^2+2-s}$.
Is it easy to see that
$f_s'(\xi) < 0$ for $s < \xi < \sqrt{s^2+2-s}$
and $f_s'(\xi) > 0$ for $\xi > \sqrt{s^2+2-s}$.
Thus, any $\xi \in (\xi_s^*,\sqrt{2s})$ satisfies
$f_s(\xi) < \max \{ f_s(\xi_s^*),f_s(\sqrt{2s}) \}$.
To prove the claim,
only $f_s(\sqrt{2s}) < f_s(\xi_s^*)$ is left to show.
We directly compute
\[
    f_s(\sqrt{2 s})
    = 2s + \frac{(2-s)^2}{2s-s^2}
    = 2s + \frac{2-s}{s}
\]
and observe that the claim
$f_s(\xi_s^*) = s^2 + \sqrt{4 (2 - s)^2 + (s^2 - 1)^2}$
is equivalent to
\begin{equation} \label{eq:xs*_def_prop}
    (\xi_s^*)^2 + \frac{(2-s)^2}{(\xi_s^*)^2-s^2}
    = f_{s}(\xi_s^*)
    = 2 (\xi_s^*)^2 - 1.
\end{equation}
The equality from left to right can be rearranged to
\[
    (\xi_s^*)^4 - (s^2 + 1) (\xi_s^*)^2 + 4 s - 4 = 0.
\]
The zeros of the in $t \in \R$ second degree polynomial
$t^2 - (s^2+1) t + 4 s - 4$
are easy to determine,
and it is straightforward to verify that $t = (\xi_s^*)^2$
is indeed one of them.

With both function values computed,
the claimed inequality $f_{s}(\sqrt{2 s}) < f_{s}(\xi_s^*)$
is equivalent to
\[
    \frac{(2 - s) (s^2 + 1)}{s}
    < \sqrt{4 (2 - s)^2 + (s^2 - 1)^2},
\]
which by squaring both sides and rearranging becomes
\[
    0
    > \frac{(2 - s)^2 (s^2 + 1)^2}{s^2} - 4 (2-s)^2 - (s^2-1)^2
    = \frac{(2-s)^2 (s^2-1)^2}{s^2} - (s^2-1)^2
    = \frac{4 (1-s) (s^2-1)^2}{s^2}.
\]
This is again true by $s \in (1,2)$.
\end{proof}

Let us point out that $\xi_s^*$ in Lemma~\ref{lem:technical}
is not in general a global maximum of $f_s$ on the interval
$[\xi_s^*,2]$.
It is therefore crucial for the proof of Theorem~\ref{thm:planar_diam} below
to use Proposition~\ref{prop:john_rounding} in its form given here with the John asymmetry.
Simply using the fact that $K \subset 2 \B^2$ would not be enough to execute the proof in its presented form.

The proof of Theorem~\ref{thm:planar_diam} is split into two parts based on the Euclidean norms of the points in a diametrical pair.
If both points have small enough norm, then Proposition~\ref{prop:john_vector_prop}~(ii) sufficiently upper bounds their distance.
If instead one point has too large norm, then we want to restrict where the other point in the diametrical pair can be.
This is done by using Proposition~\ref{prop:john_vector_prop}~(i),
based on which we can show that there must be a point from a John decomposition in a specific part of the unit circle.
A similar technique to the last part has been used for \cite[Lemma~4.1]{GrKo}.

\begin{proof}[Proof of Theorem~\ref{thm:planar_diam}]
For $s_J(K) \in \{1,2\}$, all claims follow from Theorem~\ref{thm:inner}.
(It can be easily verified that the upper bounds and
equality conditions are the same in Theorems~\ref{thm:inner}~and~\ref{thm:planar_diam} in these two cases.)
Thus, we may assume $s_J(K) \in (1,2)$.
Let $x,y \in K$ with $\gauge{x-y} = D(K)$ such that $\xi \coloneqq \gauge{x} \geq \gauge{y}$.
By Proposition~\ref{prop:john_vector_prop}~(ii), we have
\[
    D(K)
    = \gauge{x-y}
    \leq \sqrt{2 \xi^2 + 4},
\]
so the claimed inequality is true if
\[
    \xi
    \leq \xi^*
    \coloneqq \sqrt{\frac{D_{s_J(K)}^2}{2} - 2}
    = \sqrt{ \frac{s_J(K)^2 + 1 + \sqrt{4 (2 - s_J(K))^2 + (s_J(K)^2 - 1)^2}}{2} }.
\]
Proposition~\ref{prop:john_vector_prop}~(ii) further shows for this case that equality holds in $\gauge{x-y} \leq D_{s_J(K)}$ if and only if
\begin{equation} \label{eq:planar_diam_eq_cond}
    \gauge{x} = \gauge{y} = \xi^*
    \quad \text{and} \quad
    x^T y = - 2.
\end{equation}
Using $\gauge{x+y}^2 = \gauge{x}^2 + \gauge{y}^2 + 2 x^T y$,
we see that this is equivalent to
\[
    \gauge{x} = \gauge{y} = \xi^*
    \quad \text{and} \quad
    \gauge{\frac{x+y}{2}} = \sqrt{ \frac{(\xi^*)^2 - 2}{2} }.
\]
Thus, the equality conditions on $x$ and $y$ follow in this case.
Lastly, an easy computation verifies that Example~\ref{ex:inner_small_asym} provides some
$K \in \CK^2$ in John position with $D(K) = D_{s_J(K)}$.

From now on, we assume $\xi > \xi^*$.
Our goal is to show that
$D(K) = \gauge{x-y} < D_{s_J(K)}$ in this case,
which then completes the proof.

We note that Lemma~\ref{lem:technical} shows
$\xi^* > s_J(K)$,
so in particular $\xi > s_J(K)$.
Moreover, $x \in K \subset \sqrt{2 s_J(K)} \, \B^2$
by Proposition~\ref{prop:john_rounding}
yields $\xi = \gauge{x} \leq \sqrt{2 s_J(K)} < 2$.
Applying an appropriate rotation and then
reflecting everything at the $x$-axis if necessary,
we may assume $x = (\xi,0)$ and that $y$ does not lie below the $x$-axis.
Further, let $u^1, \ldots, u^m \in \bd(K) \cap \bd(\B^2)$ and
$\lambda_1, \ldots, \lambda_m > 0$ form a John decomposition.
With $P \coloneqq \{ v \in \R^2 : v^T u^i \leq 1 \text{ for all } i \in [m] \}$,
we have $K \subset P$ and $\jel(P) = \B^2$.

Next, we define
\[
    \alpha = - \frac{2}{\xi} < -1,
    \quad
    \beta = \frac{2 s_J(K) - \xi^2}{\xi \sqrt{\xi^2 - s_J(K)^2}} \geq 0,
    \quad \text{and} \quad
    p = (\alpha,\beta) \in \R^2 \setminus \B^2.
\]
For the Euclidean unit vectors
\[
    q^1
    \coloneqq \left(
        - \frac{s_J(K)}{\xi},
        - \frac{\sqrt{\xi^2 - s_J(K)^2}}{\xi}
    \right)
    \quad \text{and} \quad
    \tilde{q}^1
    \coloneqq \left(
        - \frac{s_J(K)}{\xi},
        \frac{\sqrt{\xi^2 - s_J(K)^2}}{\xi}
    \right),
\]
we observe that
\[
    1
    = p^T q^1
    = p^T \tilde{q}^1 - 2 \frac{2 s_J(K) - \xi^2}{\xi^2}
    \leq p^T \tilde{q}^1.
\]
Thus, if we define $q^2$ as the other point in
$\bd(\B^2) \setminus \{q^1\}$ satisfying $p^T q^2 = 1$,
then $\tilde{q}^1$ lies on the minor circular arc
connecting $q^1$ and $q^2$.

\begin{figure}[ht]
\def\scale{2}
\def\ds{1.5/\scale pt}
\def\s{1.6}
\def\d{1.7}
\def\a{(-2/\d)}
\def\b{((2*\s-\d^2)/(\d*sqrt(\d^2-\s^2)))}
\def\zx{((sqrt((\d^2-1)*(\d^2-\s^2))-\s)/\d)}
\def\zy{((sqrt(\d^2-\s^2)+\s*sqrt(\d^2-1))/\d)}
\def\psipx{(-\s/\d + sqrt(\d^2-\s^2)/\d * sqrt(\a^2+\b^2-1))}
\def\psipy{(sqrt(\d^2-\s^2)/\d + \s/\d * sqrt(\a^2+\b^2-1))}
\begin{tikzpicture}[scale=\scale]
\draw (0,0) circle(1);
\draw[dashed] circle(\d);

\draw[red] ({(\a+\b*sqrt(\a^2 + \b^2 - 1))/(\a^2 + \b^2)},{(\b-\a*sqrt(\a^2 + \b^2 - 1))/(\a^2 + \b^2)})
		arc ({90-atan2({(\a+\b*sqrt(\a^2 + \b^2 - 1))/(\a^2 + \b^2)},{(\b-\a*sqrt(\a^2 + \b^2 - 1))/(\a^2 + \b^2)})}
		:{90-atan2({-\s/\d},{sqrt(\d^2-\s^2)/\d})}
		:1);

\draw[dotted] ({-2/\d},-\d) -- ({-2/\d},\d);

\draw[dashdotted] ({\a},{\b}) -- ({-\s/\d},{-sqrt(\d^2-\s^2)/\d});
\draw[dashdotted] ({\a},{\b}) -- ({(\a+\b*sqrt(\a^2 + \b^2 - 1))/(\a^2 + \b^2)},{(\b-\a*sqrt(\a^2 + \b^2 - 1))/(\a^2 + \b^2)});
\draw[dashdotted] ({-\d/\s},0) -- ({(sqrt((\d^2-1)*(\d^2-\s^2))-\s)/\d},{(sqrt(\d^2-\s^2)+\s*sqrt(\d^2-1))/\d});
\draw[dashdotted] ({\zx},{\zy}) -- ({(\zx+\zy*sqrt(\d^2 - 1))/\d^2},{(\zy-\zx*sqrt(\d^2 - 1))/\d^2});
\draw[dashdotted] ({\psipx},{\psipy}) -- ({(\psipx+\psipy*sqrt(\a^2+\b^2-1))/(\a^2+\b^2)},{(\psipy-\psipx*sqrt(\a^2+\b^2-1))/(\a^2+\b^2)});

\fill (\d,0) circle(\ds) node[anchor=west] {$x$};
\fill ({-\d/\s},0) circle(\ds) node[anchor=west] {$\frac{- x}{s_J(K)}$};

\fill ({\a},{\b}) circle(\ds) node[anchor=east] {$p$};
\fill ({\psipx},{\psipy}) circle(\ds) node[anchor=east] {$\psi(p)$};

\fill ({\zx},{\zy}) circle(\ds) node[anchor=south] {$z$};

\fill ({(\a+\b*sqrt(\a^2 + \b^2 - 1))/(\a^2 + \b^2)},{(\b-\a*sqrt(\a^2 + \b^2 - 1))/(\a^2 + \b^2)}) circle(\ds)
	node[anchor=north west] {$q^2$};
\fill ({(\psipx+\psipy*sqrt(\a^2+\b^2-1))/(\a^2+\b^2)},{(\psipy-\psipx*sqrt(\a^2+\b^2-1))/(\a^2+\b^2)}) circle(\ds)
	node[anchor=north] {$\psi(q^2)$};
\fill ({-\s/\d},{sqrt(\d^2-\s^2)/\d}) circle(\ds) node[anchor=west] {$\tilde{q}^1$};
\fill ({-\s/\d},{-sqrt(\d^2-\s^2)/\d}) circle(\ds) node[anchor=west] {$q^1$};
\fill ({(\zx+\zy*sqrt(\d^2 - 1))/\d^2},{(\zy-\zx*sqrt(\d^2 - 1))/\d^2}) circle(\ds) node[anchor=south] {$q^3$};

\fill (0,0) circle(\ds) node[anchor=west] {$0$};
\end{tikzpicture}
\caption{
The points defined in the proof of Theorem~\ref{thm:planar_diam} for $s_J(K) = \s$ and $\xi = \d$.
The solid (resp.~dashed) circle has radius $1$ (resp.~$\xi$). The dotted line consist of all points $v \in \R^2$ with $v_1 = - \frac{2}{\xi}$.
One of the $u^i$ must lie on the minor circular arc connecting $\tilde{q}^1$ and $q^2$ (red).
}
\label{fig:planar_diam}
\end{figure}
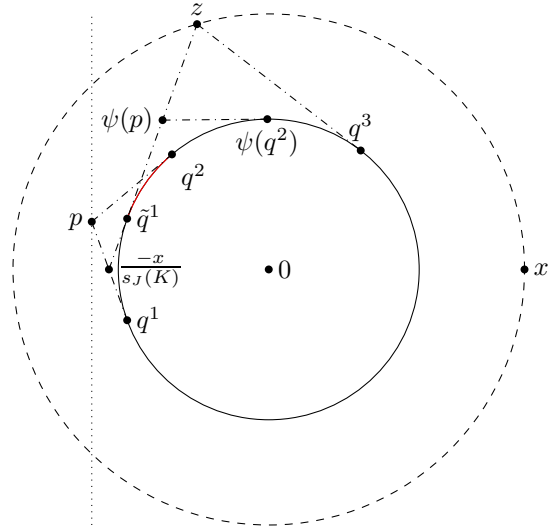

Now, we claim that one of the $u^i$ must lie on
the minor circular arc connecting $\tilde{q}^1$ and $q^2$.
We assume for a proof by contradiction the contrary.
Since $x \in K$,
we have by the definition of $s_J(K)$ that
$\frac{-x}{s_J(K)} = (-\frac{\xi}{s_J(K)},0) \in K$.
The points $q^1$ and $\tilde{q}^1$ are the intersection points of the tangents from $\frac{-x}{s_J(K)}$ to $\B^2$ with $\B^2$.
Thus, all points on the minor circular arc connecting $q^1$ and $\tilde{q}^1$,
except $q^1$ and $\tilde{q}^1$ themselves, lie in $\inte( \{ \frac{-x}{s_J(K)} \} \cup \B^2) \subset \inte(K)$.
Together with the contradiction assumption,
we see that the only point on the minor circular arc connecting $q^1$ and $q^2$ that could be one of the $u^i$ is $q^1$.
Since any $u \in \bd(\B^2)$ outside this circular arc satisfies $u^T p < 1$,
we obtain that any $u^i \neq q^1$ satisfies $(u^i)^Tp < 1$.
Now, let $\varepsilon > 0$ and define $p^\varepsilon \coloneqq p + \varepsilon (p-q^1)$.
There exist only finitely many $u^i$, so $(u^i)^Tp^\varepsilon < 1$ for every $u^i \neq q^1$ if $\varepsilon$ is sufficiently small.
Additionally, $p$ being on the tangent to $\B^2$ at $q^1$ yields $(q^1)^Tp^\varepsilon = 1$.
In summary, $(u^i)^Tp^\varepsilon \leq 1$ for all $u^i$ and thus $p^\varepsilon \in P$.
However, we now have a contradiction to Proposition~\ref{prop:john_vector_prop}~(i)
by $x,p^\varepsilon \in P$ and $\jel(P) = \B^2$ yet $x^T p^\varepsilon < x^T p = -2$.
Altogether, there must exist some $u^i$ on the minor circular arc connecting $\tilde{q}^1$ and $q^2$, say $u^1$.

Next, we observe that
\[
    \gauge{p}^2
    = \alpha^2 + \beta^2
    = \frac{4}{\xi^2} + \frac{4 s_J(K)^2 + \xi^4 - 4 s_J(K) \xi^2}
                                {\xi^2 (\xi^2 - s_J(K)^2)}
    = \frac{4 + \xi^2 - 4 s_J(K)}{\xi^2 - s_J(K)^2}
    = 1 + \frac{(2-s_J(K))^2}{\xi^2 - s_J(K)^2}.
\]
By $\xi > \xi^*$ and \eqref{eq:xs*_def_prop},
we have
\[
    1 + \frac{(2-s_J(K))^2}{\xi^2-s_J(K)^2}
    < 1 + \frac{(2-s_J(K))^2}{(\xi^*)^2-s_J(K)^2}
    = (\xi^*)^2
    < \xi^2,
\]
so $\gauge{p} < \xi$.
Thus, for $\psi \colon \R^2 \to \R^2$ the rotation
that maps $q^1$ to $\tilde{q}^1$,
we obtain that $\psi(p)$ lies on the segment connecting
$\psi(q^1) = \tilde{q}^1$
and the point
\[
    z
    = \left(
        \frac{\sqrt{(\xi^2-1) (\xi^2-s_J(K)^2)} - s_J(K)}{\xi},
        \frac{\sqrt{\xi^2-s_J(K)^2} + s_J(K) \sqrt{\xi^2-1}}{\xi}
        \right)
    \in \bd(\xi \B^2),
\]
where the latter is the intersection point of the line
$\{ v \in \R^2 : v^T \tilde{q}^1 = 1 \}$ and $\bd(\xi \B^2)$
with second coordinate larger than that of $\tilde{q}^1$.
Since $\psi(q^2)$ is a point at which
a tangent from $\psi(p)$ to $\B^2$ meets $\B^2$,
we obtain that $\psi(q^2)$ lies on the minor circular arc
connecting the intersection points of tangents from $z$ to $\B^2$,
which are $\tilde{q}^1$ and some $q^3 \in \bd(\B^2)$.
Furthermore, $q^2$ clearly lies on the minor circular arc
connecting $\tilde{q}^1$ and $\psi(q^2)$.
Altogether, $q^2$ lies on the minor circular arc connecting $\tilde{q}^1$ and $q^3$.
Therefore, independent of the precise position of $u^1$
on the minor circular arc between $\tilde{q}^1$ and $q^2$,
we have $(u^1)^Tp \geq 1$ and $(u^1)^T z \geq 1$.
Consequently, no point in $\{ \lambda z + \mu p \in K : \lambda, \mu \geq 0 \}$
can lie outside $\conv(\{0,p,z\})$,
as it would otherwise have inner product larger than $1$ with $u^1$.

In summary, we have by $\gauge{y} \leq \gauge{x}$
and Proposition~\ref{prop:john_vector_prop}~(i)
that $y$ lies somewhere in
\[
    \left\{ v \in \xi \B^2 : v_1 \geq z_1 \right\}
    \cup \conv(\{p, (p_1,0), z, (z_1,0) \}).
\]
Any point $v$ in the first set satisfies
\[
    \gauge{x-v}^2
    = \gauge{x}^2 + \gauge{v}^2 - 2 x^T v
    \leq \gauge{x}^2 + \xi^2 - 2 x_1 z_1
    = \gauge{x}^2 + \gauge{z}^2 -2 x^T z
    = \gauge{x-z}^2,
\]
so $z$ maximizes the Euclidean distance to $x$
among all points in the first set.
Clearly, one of $z$ and $p$ maximizes the distance to $x$
for the second set.
Finally, we compute
\begin{align*}
    \gauge{x-z}^2
    & = \gauge{x}^2 + \gauge{z}^2 - 2 z^T x
    = 2 (\xi^2 - \sqrt{(\xi^2-1) (\xi^2-s_J(K)^2)} + s_J(K))
    \\
    & = 2 \left( s_J(K)^2 + \sqrt{\xi^2 - s_J(K)^2} \left( \sqrt{\xi^2 - s_J(K)^2} - \sqrt{\xi^2 - 1} \right) + s_J(K) \right),
\intertext{
which, by $1 < s_J(K) < \xi \leq 2$, can be further upper bounded 
}
    & < 2 ( s_J(K)^2 + s_J(K) )
    < 2 s_J(K)^2 + 4
    = s_J(K)^2 + 5 + \sqrt{(s_J(K)^2 - 1)^2}
    < D_{s_J(K)}^2.
\end{align*}
Moreover,
\[
    \gauge{x-p}^2
    = \gauge{x}^2 + \gauge{p}^2 - 2 p^T x
    = \xi^2 + \frac{(2-s_J(K))^2}{\xi^2-s_J(K)^2} + 5,
\]
so Lemma~\ref{lem:technical} (with the notation there) shows by $\xi^* < \xi \leq \sqrt{2 s_J(K)}$ that
\[
    \gauge{x - p}
    = \sqrt{f_{s_J(K)}(\xi) + 5}
    < \sqrt{f_{s_J(K)}(\xi^*) + 5}
    = D_{s_J(K)}.
\]
Altogether, we obtain
\[
    D(K)
    = \gauge{x-y}
    \leq \max \{ \gauge{x-z}, \gauge{x-p} \}
    < D_{s_J(K)},
\]
which is what we needed to show to complete the proof.
\end{proof}

\begin{remark} \label{rem:planar_diam}
The equality case in Theorem~\ref{thm:planar_diam} shows for
$K \in \CK^2$ in John position and $x,y \in K$ that
$\gauge{x-y} = D_{s_J(K)}$
requires $x^T y = -2$
(cf.~\eqref{eq:planar_diam_eq_cond}).
By Proposition~\ref{prop:john_vector_prop}~(i),
this also means that the vectors $u^i$ in a John decomposition for $K$
must all have inner product $1$ with at least one of $x$ and $y$.
In other words, only the points that are obtained as
intersection points of $\B^2$ with
the tangents from $x$ and $y$ to $\B^2$
can occur in a John decomposition for $K$.
This shows that the inner product of $c = \frac{x+y}{2}$
with the $u^i$ can only take the values
\[
\tau_{2,1}(\mu_{2,1}(s_J(K))) \sqrt{\mu_{2,1}(s_J(K))-2} = \frac{\gauge{c}^2}{s_J(K)-1}
\quad \text{and} \quad
-\frac{\sqrt{\mu_{2,1}(s_J(K))-2}}{2 \tau_{2,1}(\mu_{2,1}(s_J(K)))} = \frac{1-s_J(K)}{2}
\]
as mentioned in the discussion above Lemma~\ref{lem:mu_tau_prop}.
One way to see this without further computation is to first observe that this claim is true for the body provided in Example~\ref{ex:inner_small_asym},
where we use Lemma~\ref{lem:mu_tau_prop}~(iv) for the identity for the left-hand value and the computations in the proof of Example~\ref{ex:inner_small_asym} for the right-hand one.
Second, we note that the equality conditions in
Theorem~\ref{thm:planar_diam} show for any $s \in [1,2]$ that
the pair of vectors $x,y$ with $\gauge{x-y} = D_s$
that can be contained in a convex body $K \in \CK^2$
with $\jel(K) = \B^2$ and $s_J(K) = s$
is unique up to rotation around the origin.
Thus, the computation of $c^T u^i$ in the general case
can be reduced to Example~\ref{ex:inner_small_asym}.
\end{remark}

\section{Affine Width-Circumradius- and Diameter-Inradius-Ratios}
\label{sec:aff_ratios}

In this last section, we focus on bounding the volume of $1$-dimensional ellipsoids, i.e., the length of segments, that are contained in convex bodies.
Applying the results from the previous sections, we obtain inequalities for the width-circumradius-ratio $w/R$ and the diameter-inradius-ratio $D/r$ of convex bodies.
An upper bound to the first ratio and a lower bound to the second one,
both based on the Minkowski asymmetry,
have been obtained in \cite[Corollary~6.4]{BrKoSharp} (see also Corollaries~\ref{cor:aff_w_R_ratio}~and~\ref{cor:aff_d_r_ratio} below).
It is in general not possible to give non-trival bounds in the other direction for either ratio by involving only affinely independent parameters.
This is easy to see from affinely transforming any convex body such that its width (and hence inradius) remains bounded but its diameter (and hence circumradius) becomes arbitrarily large..

There are many other examples of ratios between
geometric functionals that can in general be bounded
only trivially due to similar problems with affine transformations.
These issues can often be circumvented by instead asking if
at least some affine transformation exists for every convex body
such that we can obtain meaningful bounds.
Examples of results in this direction include
Ball's reverse isoperimetric inequality \cite{BallRev}
and works on the isodiametric and isominwidth problem
\cite{Be,GoSh}.
A marked difference to the problems considered
here is that the volume of the convex body itself is not involved in our case.
Instead, we work along the lines of Jung, Steinhagen, Bohnenblust, and Leichtweiss by comparing different radial functionals.

Since measuring the length of segments does not inherently require the use of Euclidean distances, it is natural to consider the problems of bounding the $w/R$- and $D/r$-ratios in more general settings.
We therefore first consider the case where $\B^n$ is replaced by an arbitrary gauge body $C \in \CK^n$,
for which we require some heavier notation.
For $X,Y \subset \R^n$, we write $X \subset_t Y$ to mean that some translate of $X$ is a subset of $Y$.
Let the four radial functionals be defined in the standard way by
\begin{alignat*}{2}
    R(K,C) & \coloneqq \inf \{ \rho \geq 0 : K \subset_t \rho C \},
    &
    r(K,C) & \coloneqq \sup \{ \rho \geq 0 : \rho C \subset_t K \},
    \\
    w(K,C) & \coloneqq \inf_{a \in \R^n \setminus \{0\}} 2 \frac{h_K(a) + h_K(-a)}{h_C(a) + h_C(-a)},~
    &
    D(K,C) & \coloneqq \sup_{a \in \R^n \setminus \{0\}} 2 \frac{h_K(a) + h_K(-a)}{h_C(a) + h_C(-a)}.
\end{alignat*}
See, e.g., \cite{BrGorRD,BrKoSharp} for some basic properties of these functionals.

We shall further use the two following notions.
For $K,C \in \CK^n$ we denote their \cemph{Banach-Mazur distance} by
\begin{align*}
    d_{BM}(K,C)
    \coloneqq & \inf \{ \rho \geq 0 : C \subset_t AK \subset_t \rho C, A \in GL_n(\R) \}
\intertext{
and their \cemph{Gr\"unbaum distance} by
}
    d_{GR}(K,C)
    \coloneqq & \inf \{ \vert \rho \vert : C \subset_t AK \subset_t \rho C, A \in GL_n(\R), \rho \in \R \}.
\end{align*}
It is shown in \cite[Theorem~5.1]{GLMP} that
$d_{GR}(K,C) \leq n$.
Obviously, we always have $d_{GR}(K,C) \leq d_{BM}(K,C)$ and equality if one of $K$ and $C$ is symmetric.
For the sake of familiarity, we shall use the more well-known Banach-Mazur distance whenever possible.

\begin{theorem} \label{thm:general_aff_ratios}
For $K,C \in \CK^n$, we have
\[
    \frac{s(K)+1}{s(C)+1} \max \left\{ \frac{s(C)}{s(K)}, 1 \right\}
    \leq \left( \max_{A \in GL_n(\R)} \frac{w(A C,K)}{2 R(A C,K)} \right)^{-1}
    = \min_{A \in GL_n(\R)} \frac{D(A K,C)}{2 r(A K,C)}
    \leq d_{GR}(K,C).
\]
Equality in the lower bound is attained if there exists $\lambda \in [0,1]$ such that $K = C - \lambda C$ or $C = K - \lambda K$.
Equality in the upper bound is attained if both $K$ and $C$ are symmetric or if $C$ is a parallelotope.
In particular, the lower and upper bound are best possible for any prescribed $s(K), s(C) \in [1,n]$ and $d_{GR}(K,C) \in [1,n]$, respectively.
\end{theorem}
\begin{proof}
We first observe for $A \in GL_n(\R)$ that the results in \cite{BrKoSharp} show
\[
    \frac{D(A K,C)}{2 r(A K,C)}
    = \frac{2 R(C,A K)}{w(C,A K)}
    = \frac{2 R(A^{-1} C,K)}{w(A^{-1} C,K)}.
\]
This already suffices to prove
\[
    \left( \max_{A \in GL_n(\R)} \frac{w(A C,K)}{2 R(A C,K)} \right)^{-1}
    = \min_{A \in GL_n(\R)} \frac{D(A K,C)}{2 r(A K,C)}.
\]
For the rest of the proof, we shall focus on the right-hand term.

There exist $A \in GL_n(\R)$ and $\sigma \in \{1,-1\}$ such that
\[
    C
    \subset_t A K
    \subset_t d_{GR}(K,C) (\sigma C).
\]
In this case, we have $r(A K,C) = 1$ and
$R(A K,\sigma C) = d_{GR}(K,C)$.
It follows
\[
    \frac{D(A K,C)}{2 r(A K,C)}
    = \frac{D(A K,\sigma C)}{2 r(A K,C)}
    \leq \frac{R(A K,\sigma C)}{r(A K,C)}
    = d_{GR}(K,C),
\]
which proves the upper bound in the theorem.
To prove its claimed equality cases,
we observe that if both $K$ and $C$ are symmetric or if $C$ is a parallelotope,
then we have for any $A \in GL_n(\R)$ that
$D(A K,C) = 2 R(A K,C)$
(cf.~\cite[Theorem~1.5]{BrGorRD} and the discussion preceding it).
Since at least one body is symmetric in these cases,
\[
    \min_{A \in GL_n(\R)} \frac{D(A K,C)}{2 r(A K,C)}
    = \min_{A \in GL_n(\R)} \frac{R(A K,C)}{r(A K,C)}
    = d_{BM}(K,C)
    = d_{GR}(K,C)
\]
is true.
Since any simplex $T \in \CK^n$ and parallelotope $C \in \CK^n$ satisfy $d_{GR}(T,C) = n$ \cite[Corollary~5.8]{GLMP},
it is easy to see that continuously deformining $T$ into $C$
gives for any $d \in [1,n]$ a convex body $K \in \CK^n$ with
$d_{GR}(K,C) = d$.

Toward the lower bound, we note that the inequality
\[
    (s(K)+1) r(K,C)
    \leq s(C) r(K,-C) + R(K,-C)
    \leq \frac{s(C)+1}{2} D(K,-C)
    = \frac{s(C)+1}{2} D(K,C)
\]
is proved in \cite[Theorem~1.1]{BrGo}.
Using \cite[Theorem~6.1]{BrKoSharp},
which in particular states $\frac{R(K,C)}{r(K,C)} \geq \frac{s(C)}{s(K)}$,
we additionally obtain
\[
    \frac{s(C)}{s(K)} (s(K)+1) r(K,C)
    \leq s(C) r(K,C) + R(K,C)
    \leq \frac{s(C)+1}{2} D(K,C).
\]
Rearranging the two above inequalities shows the lower bound in the theorem.

For the claimed equality case of the lower bound,
we first assume that
there exists $\lambda \in [0,1]$ with $C= K - \lambda K$.
Then
\[
    \frac{C-C}{2}
    = \frac{(K - \lambda K) - (K - \lambda K)}{2}
    = \frac{(K-K) + \lambda (K-K)}{2}
    = (1 + \lambda) \frac{K-K}{2}
\]
and consequently
\[
    D(K,C)
    = D \left( K, (1+\lambda) \frac{K-K}{2} \right)
    = \frac{2}{1 + \lambda}.
\]
Moreover, $-K \subset_t s(K) K$ shows
\[
    C
    = K - \lambda K
    \subset_t K + \lambda s(K) K
    = (1 + \lambda s(K)) K
\]
and thus
\[
    r(K,C)
    \geq \frac{1}{1 + \lambda s(K)}.
\]
Lastly, we estimate $s(C)$, for which we claim that
\begin{equation} \label{eq:s(C)_est}
    K - \lambda K
    \subset_t \frac{s(K) + \lambda}{1 + \lambda s(K)} (\lambda K - K).
\end{equation}
This containment is equivalent to
\[
    (1 + \lambda s(K)) (K - \lambda K)
    \subset_t (s(K) + \lambda) (\lambda K - K),
\]
for which it suffices to show
\[
    K - \lambda^2 s(K) K
    \subset_t \lambda^2 K - s(K) K.
\]
Now,
\[
    K - \lambda^2 s(K) K
    = \lambda^2 K + (1-\lambda^2) K - \lambda^2 s(K) K
    \subset_t \lambda^2 K - (1-\lambda^2) s(K) K - \lambda^2 s(K) K
\]
verifies \eqref{eq:s(C)_est}.
This yields
\[
    s(C)
    \leq \frac{s(K) + \lambda}{1 + \lambda s(K)}.
\]
Altogether, we obtain
\[
    \frac{D(K,C)}{2 r(K,C)}
    \leq \frac{\frac{2}{1+\lambda}}{\frac{2}{1+\lambda s(K)}}
    = \frac{s(K)+1}{(s(K)+1) \frac{1+\lambda}{1+\lambda s(K)}}
    = \frac{s(K)+1}{\frac{s(K)+\lambda}{1+\lambda s(K)} + 1}
    \leq \frac{s(K)+1}{s(C)+1}.
\]
By the lower bound in the theorem,
we must have equality from left to right
and thus equality in the lower bound.
Additionally, we must have
\[
    s(C)
    = \frac{s(K) + \lambda}{1 + \lambda s(K)},
\]
where it is clear that the right-hand side takes any value between $1$ and $s(K)$ for appropriate $\lambda \in [0,1]$.

Now, assume that $K \coloneqq C - \lambda C$ for some $\lambda \in [0,1]$.
Then $-C \supset_t \frac{1}{s(C)} C$ implies
\[
    K
    = C - \lambda C
    \supset_t C + \frac{\lambda}{s(C)} C
    = \left( 1 + \frac{\lambda}{s(C)} \right) C,
\]
so
\[
    r(K,C) \geq 1 + \frac{\lambda}{s(C)}.
\]
Moreover, the above computations yield
\[
    D(K,C)
    = 2 (1 + \lambda)
\]
and
\[
    s(K)
    = \frac{s(C) + \lambda}{1 + \lambda s(C)} \in [1,s(C)].
\]
Put together, we obtain
\[
    \frac{D(K,C)}{2 r(K,C)}
    \leq \frac{2 (1+\lambda)}{2 (1+\frac{\lambda}{s(C)})}
    = s(C) \frac{1+\lambda}{s(C) + \lambda}
    = \frac{s(C)}{s(K) (s(C)+1)} \frac{(s(C)+1) (1+\lambda)}{1 + \lambda s(C)}
    = \frac{s(C) (s(K)+1)}{s(K) (s(C)+1)}.
\]
This concludes the proof.
\end{proof}

\begin{remark}
Let us briefly consider the absolute bounds obtained from Theorem~\ref{thm:general_aff_ratios}.
The lower bound in the theorem can be written as
\[
    \frac{(s(K)+1) s(C)}{(s(C)+1) \min\{ s(K),s(C) \} }
    = \frac{s(C) s(K) + s(C)}{s(C) \min\{ s(K),s(C) \} + \min\{ s(K),s(C) \}}.
\]
Therefore, the theorem yields in particular the absolute bounds
\begin{equation} \label{eq:general_aff_ratios_absolute bounds}
    1
    \leq \left( \max_{A \in GL_n(\R)} \frac{w(A C,K)}{2 R(A C,K)} \right)^{-1}
    = \min_{A \in GL_n(\R)} \frac{D(A K,C)}{2 r(A K,C)}
    \leq n,
\end{equation}
both of which are best possible.

The proof of Theorem~\ref{thm:general_aff_ratios} shows that $\frac{s(C) + d_{BM}(K,C)}{s(C)+1}$ is also a valid lower bound, which tightens the lower bound in the theorem in case $s(C) \geq s(K)$.
Consequently, equality is attained in the lower bound in \eqref{eq:general_aff_ratios_absolute bounds} if and only if $K$ and $C$ are affinely equivalent.

Equality in the upper bound in \eqref{eq:general_aff_ratios_absolute bounds} requires $d_{GR}(K,C) = n$.
It is conjectured in \cite{HuNa} that the latter is true only if one of $K$ and $C$ is a simplex.
The conjecture has been partially verified:
in \cite{HuNa} under the additional assumption that one of $K$ and $C$ is smooth or strictly convex,
and in \cite{Ko2d} for the planar case when one of $K$ and $C$ is symmetric.
We therefore conjecture that equality holds in the upper bound in \eqref{eq:general_aff_ratios_absolute bounds}
only if one of $K$ and $C$ is a simplex.
Lastly, note that the Euclidean case (Proposition~\ref{prop:ball_thm}) already shows that one of $K$ and $C$ being a simplex is not always sufficient for equality in the upper bound in
\eqref{eq:general_aff_ratios_absolute bounds}.
\end{remark}

We return to the $w/R$- and $D/r$-ratios in the Euclidean case.
It turns out that the John position already allows us to establish tight bounds,
partially better than those in Theorem~\ref{thm:general_aff_ratios}.
The first corollary is obtained from combining Theorem~\ref{thm:outer} with \cite[Corollary~6.4]{BrKoSharp}.

\begin{corollary} \label{cor:aff_w_R_ratio}
Let $K \in \CK^n$.
Then
\[
    \frac{1}{\sqrt{n}}
    \leq \max_{A \in GL_n(\R)} \frac{w(AK)}{2 R(AK)}
    \leq \begin{cases}
        \min \{ \frac{\sqrt{n}}{s(K)}, \frac{s(K)+1}{2 s(K)} \}, & \text{if $n$ odd},
        \\
        \min \{ \frac{n+1}{s(K) \sqrt{n+2}}, \frac{s(K)+1}{2 s(K)} \}, & \text{if $n$ even}.
    \end{cases}
\]
The upper bound is tight for any prescribed $s(K) \in [1,n]$.
The lower bound is attained for cross-polytopes and parallelotopes in all dimensions,
as well as for simplices in odd dimensions.
\end{corollary}
\begin{proof}
The upper bound and its tightness are obtained in
\cite[Corollary~6.4]{BrKoSharp}.
The lower bound follows from Theorem~\ref{thm:outer}.
Since upper and lower bound coincide if $n$ is odd and $s(K)=n$,
we have equality in the lower bound for simplices
in odd dimensions.
Lastly, Theorem~\ref{thm:general_aff_ratios} yields equality in the lower bound for cross-polytopes and parallelotopes
in any dimension $n$ by the well-known fact that their Banach-Mazur distance to $\B^n$ is equal to $\sqrt{n}$.
\end{proof}

Let us point out for the family of bodies from Theorem~\ref{thm:outer} that equality in the inequality in that theorem is not necessarily enough to obtain equality in the lower bound in the corollary above.
In general, other affine transformations unrelated to the John position may achieve larger $w/R$-ratios.

Let us further mention that
the lower bound in the corollary is not sharp for simplices in even dimensions.
It is shown in \cite{Al} that regular simplices $T$ uniquely maximize the $w/R$-ratio among all simplices with
\[
    \max_{A \in GL_n(\R)} \frac{w(AT)}{2 R(AT)}
    = \frac{n+1}{n \sqrt{n+2}}
    > \frac{1}{\sqrt{n}}.
\]
This suggests that one may sharpen the lower bound in this case by involving an asymmetry measure.

Now, we turn to the $D/r$-ratio.
First, we again collect known results from
Theorems~\ref{thm:inner}~and~\ref{thm:planar_diam}
as well as \cite[Corollary~6.4]{BrKoSharp}.

\begin{corollary} \label{cor:aff_d_r_ratio}
Let $K \in \CK^n$.
Then
\[
    \max \left\{ s(K) \sqrt{\frac{n+1}{2n}}, \frac{s(K)+1}{2} \right\}
    \leq \min_{A \in GL_n(\R)} \frac{D(AK)}{2 r(AK)}
    \leq \begin{cases}
        \frac{D_{s_J(K)}}{2}, & \text{if $n=2$},
        \\
        \sqrt{\frac{n (s_J(K)+1)}{2}}, & \text{else}.
    \end{cases}
\]
The lower bound is tight for any prescribed $s(K) \in [1,n]$.
The upper bound is tight for $s_J(K) \in \{1,n\}$.
\end{corollary}
\begin{proof}
The lower bound and its tightness follow from
\cite[Corollary~6.4]{BrKoSharp}.
The upper bound is obtained in
Theorems~\ref{thm:inner}~and~\ref{thm:planar_diam}.
Since upper and lower bound coincide if $s_J(K)=s(K)=n$,
which applies if $K$ is a simplex,
we have tightness of the upper bound for $s_J(K)=n$.
Lastly, Theorem~\ref{thm:general_aff_ratios} yields equality in the upper bound for cross-polytopes and parallelotopes,
i.e., tightness for $s_J(K) = 1$.
\end{proof}

Let us point out that the above upper bound is by Theorem~\ref{thm:inner}
not tight for $n \geq 3$ and $s_J(K) \in (1,1+\frac{2}{n})$.
Furthermore, the tightness in Theorems~\ref{thm:inner}~and~\ref{thm:planar_diam}
is again in general not enough
to prove tightness in the upper bound
since other affine transformations unrelated to the John position
could provide a better upper bound on the $D/r$-ratio.

The remainder of this section is devoted to a more detailed study
of the absolute upper bound following from
Corollary~\ref{cor:aff_d_r_ratio}.
The inequality itself could of course also be obtained already from Proposition~\ref{prop:ball_thm}.
The novel part is the characterization of the equality case.

\begin{theorem} \label{thm:aff_d_r_ratio}
Let $K \in \CK^n$.
Then
\[
    \min_{A \in GL_n(\R)} \frac{D(AK)}{2r(AK)}
    \leq \sqrt{\frac{n (n+1)}{2}},
\]
with equality if and only if $K$ is a simplex.
\end{theorem}

Before we give the proof,
we justify calling the above result a theorem.
To do so, we highlight some of its interplay with Jung's inequality.
Together, they yield a strengthening of the fact that simplices are the only convex bodies $K \in \CK^n$ with equality in the inequality $d_{BM}(K,\B^n) \leq n$.
This inequality plays a central role for the theory of Banach-Mazur distances, so it comes to no surprise that it has received significant attention in the literature:
Sharpenings of the inequality are obtained in \cite{BeFr,BrKoSharp}.
Its unique equality case has first been established in \cite{Le2}, rediscovered in \cite{Pa}, and broadly generalized in \cite{HuNa}.
Even a stability version is given in \cite{Ko}.
A new proof of the unique equality case is presented in the following.
Using Jung's inequality on the $R/D$-ratio,
we have for any $K \in \CK^n$ that
\begin{align}
\begin{split} \label{eq:dbm}
    d_{BM}(K,\B^n)
    & = \min_{A \in GL_n(\R)} \frac{R(AK)}{r(AK)}
    = \min_{A \in GL_n(\R)} \frac{D(AK)}{r(AK)} \frac{R(AK)}{D(AK)}
    \\
    & \leq \min_{A \in GL_n(\R)} \frac{D(AK)}{r(AK)} \sqrt{\frac{n}{2 (n+1)}}
    \leq \sqrt{2 n (n+1)} \, \sqrt{\frac{n}{2 (n+1)}}
    = n,
\end{split}
\end{align}
where equality from left to right requires by
Theorem~\ref{thm:aff_d_r_ratio} that $K$ is a simplex.
It might therefore be of interest to improve the uniqueness of
simplices as maximizers in Theorem~\ref{thm:aff_d_r_ratio}.
For example, a stability version of this result would
immediately yield a stability of simplices uniquely
maximizing the Banach-Mazur distance to $\B^n$ as well.
While the latter stability has been established in \cite{Ko},
it is unknown if it is of the best possible order.

Let us derive a second consequence of \eqref{eq:dbm}.
Since equality holds from left to right when $K$ is a simplex,
we must have equality in Jung's inequality for any affine transformation $AK$ that minimizes the $D/r$-ratio for $K$.
By the equality case in Jung's inequality, $AK$ must be regular.
Consequently,  we obtain the following corollary,
which is the analog of the result in \cite{Al}
that regular simplices uniquely maximize the width among
all simplices of a given circumradius.

\begin{corollary}
Regular simplices uniquely minimize the Euclidean diameter among all simplices of a given Euclidean inradius.
\end{corollary}

We turn to the proof of Theorem~\ref{thm:aff_d_r_ratio}.
The main idea to obtain the strict inequality for bodies that are not simplices
is to take advantage of Theorem~\ref{thm:inner}:
It either shows that transforming the body into John position already gives the strict inequality,
or it provides structural information about the body that,
together with Proposition~\ref{prop:john_vector_prop},
allows us to envelope the body with a certain polytope.
Based on this approximation, we then show that slight,
appropriate deviations from the John position reduce the $D/r$-ratio.
These slight deviations are provided by the lemma below.

\begin{lemma} \label{lem:aff_d_r_ratio_reduc}
Let $U \subset \R^n$ be a linear subspace,
$c \in U^\perp$,
$\alpha,\mu > 1$,
and let $A \colon \R^n \to \R^n$ be an affine linear map
defined by
$A(x) = \op{x}{U} + \alpha (\op{x}{U^\perp} - c) + c$.
Then
\[
    \B^n
    \subset A( \conv((\mu (\B^n \cap U) + c) \cup \B^n) )
    =: C
\]
and any common boundary point $x$ of $\B^n$ and $C$ satisfies $c^T x > 0$.
\end{lemma}
\begin{proof}
We write $B \coloneqq \B^n \cap U$ and
$B' \coloneqq \B^n \cap U^\perp$.
Let $x \in \B^n$ be chosen arbitrarily.
Then there exist $b \in B$ and
$y \in \sqrt{1 - \gauge{b}^2} \, B'$
such that $x = b + y$.
We can compute
\[
    A(x)
    = b + \alpha (y - c) + c
\]
and
\[
    A(b + c)
    = b + \alpha(c-c) + c 
    = b + c,
\]
where $b+c \in \mu B + c$ by $\mu > 1$.
Since $\alpha > 1$,
\[
    x
    = b + y
    = \frac{\alpha-1}{\alpha} (b+c) + \frac{1}{\alpha} (b+\alpha(y-c)+c)
    \in (A(b+c),A(x))
    \subset C
\]
shows $\B^n \subset C$.

Now, additionally assume $x \in \bd(C)$.
Then $x$ must be a boundary point of $\B^n$ as well and therefore $C \subset H_{(x,1)}^\leq$.
Since $c,y \in U^\perp$, $B \subset U$, and $\mu B + c = A(\mu B + c) \subset C$ we have
\[
    1
    = h_C(x)
    \geq h_{\mu B + c}(b+y)
    = \mu \gauge{b} + c^T y
    \geq \mu \gauge{b} - \gauge{c} \gauge{y}
    = \mu \gauge{b} - \gauge{c} \sqrt{1 - \gauge{b}^2}.
\]
Therefore, $\mu > 1$ shows $\gauge{b} < 1$
and thus $y \neq 0$.
Additionally, $x \in (A(b+c),A(x)) \subset C$
and $h_C(x) = x^T x = 1$
imply $1 = x^T A(b+c) = x^T A(x)$,
where $c,y \in U^\perp$ and $b \in U$ yield
\[
    \gauge{b}^2 + \gauge{y}^2
    = 1
    = x^T A(b+c)
    = x^T (b+c)
    = \gauge{b}^2 + y^T c.
\]
It follows $c^T x = c^T y = \gauge{y}^2 > 0$ as claimed.
\end{proof}

\begin{proof}[Proof of Theorem~\ref{thm:aff_d_r_ratio}]
The only remaining part to prove is the strictness of
the upper bound for convex bodies that are not simplices.
To this end, w.l.o.g.~let $K \in \CK^n$ be in John position.
By Theorem~\ref{thm:inner}, we now have
$D(K) \leq \sqrt{2 n (n+1)}$,
where we may assume that equality holds.
Our goal is to show that either $K$ is a simplex
or that there exists some affine map $A$ such that
$r(AK) > 1$ and still
$D(AK) \leq \sqrt{2 n (n+1)}$.
This then completes the proof.

By Propsition~\ref{prop:john}, there exist some vectors
$u^1, \ldots, u^m \in \bd(K) \cap \bd(\B^n)$
and weights $\lambda_1, \ldots, \lambda_m > 0$
that form a John decomposition for $K$.
We define the polytope $P \coloneqq \{ x \in \R^n : x^T u^i \leq 1 \text{ for all } i \in [m] \}$,
which satisfies $\jel(P) = \B^n$ and $K \subset P \subset n \B^n$.
Any point $x \in P$ with $\gauge{x} = n$
is an extreme point of $n \B^n$ and therefore also of $P$.
Consequently, the set $X = \{x \in P : \gauge{x} = n\}$ is finite.
For any point $x \in X \setminus K$,
there exists some halfspace $H^\leq(x)$ with
$K \subset H^\leq(x)$ but $x \notin H^\leq(x)$.
We intersect $P$ with all of these finitely many halfspaces
to obtain another polytope $Q$,
which satisfies $K \subset Q \subset P \subset n \B^n$.
Moreover, $\jel(Q) = \B^n$ and
$\bd(Q) \cap \bd(n \B^n) = \bd(K) \cap \bd(n \B^n)$.
Since $Q$ is a polytope,
there exist only finitely many pairs of vertices of $Q$
such that at least one of them has Euclidean norm less than $n$.
By Proposition~\ref{prop:john_vector_prop}~(iii),
the distance between any such pair of points
is strictly less than $\sqrt{2 n (n+1)}$.
Consequently, there exists some $M < \sqrt{2 n (n+1)}$
such that any two vertices of $Q$ of which at least one has Euclidean norm less than $n$ are at distance at most $M$ from each other.

By $D(K) = \sqrt{2 n (n+1)}$
and Proposition~\ref{prop:john_vector_prop}~(iii),
the set $\bd(K) \cap \bd(n \B^n)$
contains at least two points
and any two points in this set are at Euclidean distance $\sqrt{2 n (n+1)}$ from each other.
Thus, $\bd(K) \cap \bd(n \B^n)$ is affinely independent
and its convex hull $T$ is a regular $k$-simplex
for some $k \in [n]$.
We write $x^1, \ldots, x^{k+1}$ for the vertices of $T$.

We consider two cases based on $k$.
If $k = n$, then $T$ is a regular $n$-simplex with all vertices
in $\bd(n \B^n)$.
It is well-known that such a simplex can be written as
\[
    T
    = \bigcap_{i \in [n+1]} H_{(-x^i,n)}^\leq.
\]
By Proposition~\ref{prop:john_vector_prop}~(i) and $x^i \in K$,
any $x \in K$ satisfies $x^T (-x^i) \leq n$
for all $i \in [n+1]$.
This implies $K \subset T$.
Since $T \subset K$ is true by definition,
we conclude that $K = T$ is a simplex in case $k=n$.

Now, assume $k < n$.
Let $U \subset \R^n$ be the linear $k$-space
parallel to $\aff(T)$
and define $B \coloneqq U \cap \B^n$.
Since $T$ is a regular $k$-simplex with edge length
$\sqrt{2 n (n+1)}$,
we know from Proposition~\ref{prop:ball_thm}
for $c \coloneqq \sum_{i \in [k+1]} \frac{x^i}{k+1}$
and $\mu \coloneqq \sqrt{\frac{n (n+1)}{k (k+1)}} > 1$
that $T$ contains $\mu B + c$.
Since $\mu B + c \subset T \subset K$,
the equality conditions in Theorem~\ref{thm:inner}
imply $c \in U^\perp$.
With the affine map $A \colon \R^n \to \R^n$ as defined in
Lemma~\ref{lem:aff_d_r_ratio_reduc} for
$\alpha \coloneqq \frac{\sqrt{2 n (n+1)}}{M} > 1$,
we obtain $\B^n \subset A K$
and $\bd(\B^n) \cap \bd(A K)$ belongs entirely to some open halfspace whose boundary contains the origin. 
The well-known Halfspace Lemma characterizing Euclidean inballs
(cf.~\cite[p.~55]{BoFe})
therefore shows that $\B^n$ is not an inball of $AK$,
i.e., $r(AK) > 1$.

Finally, we prove
$D(AK) \leq D(AQ) \leq \sqrt{2 n (n+1)}$.
The first inequality follows directly from $AK \subset AQ$.
For the second inequality, if suffices to show
$\gauge{A(x)-A(y)} \leq \sqrt{2 n (n+1)}$
for every pair of vertices $x,y$ of $Q$.

If $\gauge{x} = \gauge{y} = n$,
then $Q \cap \bd(n \B^n) = K \cap \bd(n \B^n)$ implies that
$x,y$ are also vertices of $T$,
so $x,y \in c + U$.
It follows $A(x) = x$, $A(y) = y$, and therefore
\[
    \gauge{A(x)-A(y)}
    = \gauge{x-y}
    \leq \sqrt{2 n (n+1)}.
\]
If instead at least one of $x,y$ has Euclidean norm less than $n$,
we first observe that
$f \colon \R^n \to \R^n$,
$f(x) = (\op{x}{U}) + \alpha (\op{x}{U^\perp})$,
is the linear part of $A$ and satisfies $\gauge{f} \leq \alpha$.
Hence,
\[
    \gauge{A(x)-A(y)}
    = \gauge{f(x-y)}
    \leq \gauge{f} \gauge{x-y}
    \leq \alpha M
    = \sqrt{2 n (n+1)}.
\]
This completes the case for $k < n$ and therefore the proof of the theorem.
\end{proof}

Let us close this final section with some remarks.
The study of ratios of radial functionals optimized under affinity appears to be an interesting new direction,
especially for ratios where ordinarily no non-trivial bounds could be obtained.
Theorem~\ref{thm:aff_d_r_ratio} and its consequences are a first indication of what might be achievable.
Consider, for instance, the following question:
Defining for symmetric $C \in \CK^n$ the \emph{Jung constant} $j_C$ as the maximum $C$-circumradius of convex bodies of unit $C$-diameter, does there exist for any convex body $K \in \CK^n$ an affine transformation $A K$ with
$\frac{D(A K,C)}{r(A K,C)} \leq \frac{n}{j_C}$?
By Theorems~\ref{thm:general_aff_ratios}~and~\ref{thm:aff_d_r_ratio}, the answer is affirmative if $C$ is a parallelotope or an ellipsoid.
Whenever the question statement is true for some $C$, we obtain a direct strengthening of the fact that no convex body is at Banach-Mazur distance larger than $n$ to $C$.
In this case, we might also ask for which bodies $K$ there exists no affine transformation such that $\frac{D(A K,C)}{r(A K,C)} < \frac{n}{j_C}$.
This would always apply for simplices, but are there any other such bodies?
If not, then we would also obtain a positive answer to the conjecture in \cite{HuNa} that $d_{GR}(K,C) = n$ is fulfilled only if one of $K$ and $C$ is a simplex for this specific $C$.

Finally, let us provide partial answers to both of the above posed questions.
It is known that always $(s(K)+1) r(K,C) \leq D(K,C)$,
with equality if $K$ is \emph{(pseudo-)complete} with respect to symmetric $C$ (see \cite[Remark~6.3]{BrKoSharp} for the inequality and \cite{BrGojung} for the notion of (pseudo-)completeness).
Therefore, if more generally $K$ has an affine transformation $BK$ that is pseudo-complete with respect to $C$, then
\[
    \min_{A \in GL_n(\R)} \frac{D(AK,C)}{r(AK,C)}
    = \frac{D(BK,C)}{r(BK,C)}
    = s(K)+1
    \leq n+1
    \leq \frac{n}{j_C},
\]
where the right-most inequality follows from $j_C \leq \frac{n}{n+1}$ by Bohnenblust's inequality.
Equality from left to right requires $s(K) = n$,
i.e., that $K$ is a simplex.

\bigskip

René Brandenberg --
Technical University of Munich, Department of Mathematics, Germany. \\
\textbf{rene.brandenberg@tum.de}

Florian Grundbacher -- 
Technical University of Munich, Department of Mathematics, Germany. \\
\textbf{florian.grundbacher@tum.de}

\vfill\eject

\appendix
\section{Approximating Convex Bodies by Full-Dimensional Ellipsoids}
\label{app:sJ}

In this part, we discuss some differences between the Minkowski asymmetry and the John asymmetry.
In passing, we shall also provide some related details about the bounds in Proposition~\ref{prop:john_rounding}.

As mentioned in the introduction, simplices are not the only convex bodies that attain John asymmetry $n$.
This was falsely claimed in \cite{Gr} (and cited as such in \cite{BrKoSharp}),
so we correct this mistake with Example~\ref{ex:john_rounding} below.
The error seems to be based on the following:
Simplices are indeed the only convex bodies with Euclidean circumball $n \B^n$ when transformed into John position
as discussed in Section~\ref{sec:aff_ratios}.
However, this does not mean that simplices are the only convex bodies for which $n \B^n$ is the smallest \emph{origin-centered} Euclidean ball that contains the convex body when transformed into John position.

In fact, a closely related question concerning Proposition~\ref{prop:john_rounding} was considered in \cite{BrKoSharp}.
It asks whether the John asymmetry in the upper bound can be replaced by the Minkowski asymmetry.
A similar strengthening has been conjectured in \cite[Remark~4]{BeFr},
there based on the Loewner ellipsoid.
A positive answer to these questions would imply the above incorrect statement about smallest origin-centered Euclidean balls containing bodies in John position,
so Example~\ref{ex:john_rounding} provides a negative answer to these questions as well.
More precisely, the example shows that $\rho = n$ can be the smallest possible choice in Proposition~\ref{prop:john_rounding} for convex bodies $K \in \CK^n$ with $s(K) < 2$ in any dimension.
In particular, the ratio between Minkowski asymmetry and John asymmetry can be of linear order in $n$ (cf.~Remark~\ref{rem:s_vs_sJ} below).

Example~\ref{ex:john_rounding} and the discussion below it additionally verify the novel parts of Proposition~\ref{prop:john_rounding} compared to \cite[Theorem~7.1]{BrKoSharp}.
Originally, the Minkowski asymmetry was used for the lower bound.
Since $s(K) \leq s_J(K)$, using the latter as lower bound in the proposition means a strengthening of the old one.
It is immediately obtained from the chain of containments $K \subset \rho \B^n \subset \rho (-K)$.
The tightness of the strengthened lower bound is discussed after Example~\ref{ex:john_rounding}.
The tightness of the upper bound in the proposition is also not considered in \cite{BrKoSharp},
which is obtained directly from the example.

\begin{example} \label{ex:john_rounding} 
Let $s \in [1,n]$, $v \in \bd(\B^n)$, and 
\[
    K \coloneqq \conv \left( \B^n \cup \left\{ - \sqrt{n s} \, v, \sqrt{\frac{n}{s}} \, v \right\} \right).
\]
Then $K$ satisfies $\jel(K) = \B^n$,
$s_J(K) = s$,
$K \cap \bd(\sqrt{ns} \, \B^n) \neq \emptyset$,
and $s(K) = \frac{2s}{s+1}$.
Moreover, $c \coloneqq -\sqrt{n s} \frac{s-1}{3s+1} v$
is the unique \cemph{Minkowski center} of $K$,
i.e., the only point with $K-c \subset s(K) (c-K)$.
\end{example} 
\begin{proof}
We prove $\jel(K) = \B^n$ by applying
Lemma~\ref{lem:construction}
for any orthonormal basis
$v^1, \ldots, v^n \in \R^n$ with $v^n = v$,
$J = [n-1]$,
and $\tau = \sqrt{\frac{s}{n}}$,
which satisfies
\[
    \sqrt{\frac{n - (n-1)}{n ((n - 1) +1)}}
    = \frac{1}{n}
    \leq \sqrt{\frac{1}{n}}
    \leq \tau
    \leq 1.
\]
To this end, we reuse the notation in
Lemma~\ref{lem:construction}.
For $\delta, \sigma \in \{-1, 1\}^{[n-1]}$, we have
\[
    (- \sqrt{n s} \, v^n)^T a^{\delta,\tau}
    < 0
    < 1
    = \left( \sqrt{\frac{n}{s}} \, v^n \right)^T a^{\delta,\tau},
\]
as well as
\[
    \left( \sqrt{\frac{n}{s}} \, v^n \right)^T b^{\sigma,\tau}
    < 0
    < 1
    = \frac{\sqrt{n s}}{n \sqrt{\frac{s}{n}}}
    = (- \sqrt{n s} \, v^n)^T b^{\sigma,\tau}.
\]
This shows $\B^n \subset K \subset P(J,\tau)$,
so Lemma~\ref{lem:construction} yields $\jel(K) = \B^n$.

For the claim $s_J(K) = s$,
it is easy to see that $-K \subset s K$,
i.e., $s_J(K) \leq s$.
The Cauchy-Schwarz inequality further yields
$h_K(v) = \sqrt{\frac{n}{s}}$ and
$h_{-K}(v) = \sqrt{n s}$,
which shows that $-K \subset \rho K$ for $\rho \geq 0$
requires $\rho \geq s$.
This shows $s_J(K) \geq s$.
The claim $K \cap \bd(\sqrt{ns} \, \B^n) \neq \emptyset$
is obviously true.

Finally, to prove $s(K) = \frac{2s}{s+1}$,
it suffices by \cite[Theorem~2.3]{BrKoDim}
to show for $a \in \R^n$ that
\begin{equation} \label{eq:john_rounding_mink}
    h_{K-c}(a) \leq \frac{2 s}{s+1} h_{K-c}(-a),
\end{equation}
with equality for some $a^1, \ldots, a^{n+1} \in \R^n \setminus \{0\}$ satisfying $0 \in \inte(\conv(\{a^1, \ldots, a^{n+1}\}))$.
This also suffices for the uniqueness of the Minkowski center $c$:
If $c' \neq c$ would be another Minkowski center of $K$,
then $0 \in \inte(\conv(\{a^1, \ldots, a^{n+1}\})$
shows that some $a^i$ satisfies $(c-c')^T a^i > 0$.
It follows
\[
    h_{K-c'}(a^i)
    > h_{K-c}(a^i)
    = \frac{2s}{s+1} h_{K-c}(-a^i)
    > \frac{2s}{s+1} h_{K-c'}(-a^i),
\]
contradicting $K-c' \subset \frac{2s}{s+1} (c'-K)$
for the Minkowski center $c'$ of $K$.
Hence, we know that no additional Minkowski center exists
once the existence of the $a^i$ as described above is proved.

To establish \eqref{eq:john_rounding_mink},
let $\mu_1, \ldots, \mu_n \in \R$ and
$a = \sum_{i \in [n]} \mu_i v^i$.
Then
\begin{align*}
    h_{K-c}(a)
    & = \max \left\{ - \sqrt{n s} \, a^T v^n, \sqrt{\frac{n}{s}} \, a^T v^n, h_{\B^n}(a) \right\} - a^T c
    \\
    & = \max \left\{ \sqrt{n s} \frac{2 (s+1)}{3s+1} (-\mu_n), \sqrt{n s} \frac{(s+1)^2}{s (3s+1)} \mu_n, \sqrt{n s} \frac{s-1}{3s+1} \mu_n + \gauge{a} \right\}
\end{align*}
First, assume $\gauge{a} \leq \sqrt{n s} \, \vert \mu_n \vert$.
If additionally $\mu_n \leq 0$, then
\[
    h_{K-c}(a)
    = \sqrt{n s} \frac{2 (s+1)}{3s+1} (-\mu_n)
    = \frac{2s}{s+1} \sqrt{n s} \frac{(s+1)^2}{s (3s+1)} (-\mu_n)
    \leq \frac{2s}{s+1} h_{K-c}(-a),
\]
with equality in the last step if
$\gauge{a} \leq \sqrt{\frac{n}{s}} \, \vert \mu_n \vert$,
which applies in particular if
$\mu_i = 0$ for all $i \neq n$.
If instead $\mu_n \geq 0$,
we have by $(s+1)^2 \leq 4s^2$ that
\[
    h_{K-c}(a)
    \leq \sqrt{n s} \frac{4 s}{3s+1} \mu_n
    = \frac{2s}{s+1} \sqrt{n s} \frac{2 (s+1)}{3s+1} \mu_n
    = \frac{2s}{s+1} h_{K-c}(-a),
\]
with equality in the first step whenever $\gauge{a} = \sqrt{n s} \, \mu_n$.
Lastly, if $\gauge{a} > \sqrt{n s} \, \vert \mu_n \vert$, then
\begin{align*}
    h_{K-c}(a)
    & = \sqrt{n s} \frac{s-1}{3s+1} \mu_n + \gauge{a}
    < \frac{4s}{3s+1} \gauge{a}
    \leq \left( \frac{2s}{s+1} - \frac{s-1}{3s+1} \right) \gauge{a}
    \\
    & < \frac{2s}{s+1} \left( \sqrt{n s} \frac{s-1}{3s+1} \mu_n + \gauge{a} \right)
    \leq \frac{2s}{s+1} h_{K-c}(-a).
\end{align*}
Altogether, the required inequality \eqref{eq:john_rounding_mink} is established.
We also showed equality in \eqref{eq:john_rounding_mink} whenever $a = - v^n$ or $a \in \bd(\B^n)$ with $a^T v^n = \frac{1}{\sqrt{ns}}$,
so the proof is complete as discussed above.
\end{proof}

Let us now briefly discuss the tightness of the lower bound in Proposition~\ref{prop:john_rounding}. The above example shows for any $v \in \bd(\B^n)$ that $\conv(\B^n \cup \{nv\})$ is in John position.
Since $K = \conv(\B^n \cup \{sv\})$ for $s \in [1,n]$ is a subset of this convex body and $\B^n \subset K$, we obtain $\jel(K) = \B^n$ as well. It is clear from the definition that $K \subset s \B^n$, and $s_J(K) = s$ is easy to see as well.

\begin{remark} \label{rem:s_vs_sJ}
Example~\ref{ex:john_rounding} yields results on the extremal behavior of the Minkowski asymmetry and Minkowski centers
when compared to their John counterparts.
The example shows that the Minkowski asymmetry of a convex body $K \in \CK^n$
can be as small as $\frac{2 s_J(K)}{s_J(K)+1} < 2$ for any prescribed $s_J(K) \in [1,n]$.
In particular, the ratio $\frac{s_J(K)}{s(K)}$ can be as large as $\frac{n+1}{2}$,
which is of linear order in $n$ (the largest possible order).

The situation for the norm of the Minkowski centers is similar as for the Minkowski asymmetry.
Example~\ref{ex:john_rounding} applied for $s=n$ shows that the Euclidean norm of the (unique) Minkowski center of a convex body in John position can be as large as $n \frac{n-1}{3n+1}$,
which behaves asymptotically like $\frac{n}{3}$.
In particular, the norm is of linear order in $n$.
This order is by Proposition~\ref{prop:john_rounding} again the largest possible.
Our observation parallels a result in \cite{Huang},
where it is shown that the norm of the center of mass of a convex body in John position can also be of linear order in $n$.
\end{remark}

\newpage
\section{Technical Analysis of Functions}
\label{app:technical}

We provide the technical details for the proof of Lemma~\ref{lem:mu_tau_prop} that were omitted in Section~\ref{sec:inner}.

\begin{proof}[Proof of Lemma~\ref{lem:mu_tau_prop}]
For (i), we observe that $\mu_{n,k}$ is differentiable in $s \in (1,1+\frac{2}{n})$ with derivative
\[
    \mu_{n,k}'(s)
    = \frac{n}{8 (k+1)} \left( 4k + 2 n (s-1)
    + \frac{2}{n}
    \frac{2 n (s-1) (3ks+k+2)-8ks+n^2 (s-1)^3}{\zeta_{n,k}(s)} \right).
\]
Now, $\mu_{n,k}'(s) > 0$ is equivalent to
\begin{equation} \label{eq:mu_deriv_pos}
    (4k + 2 n (s-1)) \zeta_{n,k}(s) + \frac{2}{n} \left( 2 n (s-1) (3ks+k+2)-8ks+n^2 (s-1)^3 \right) > 0.
\end{equation}
From the definition of $\zeta_{n,k}(s)$,
it is clear that
\[
    \zeta_{n,k}(s)
    \geq \sqrt{\left( 2 \left( 1+\frac{2}{n}-s \right) - (s^2-1) \right)^2}
    \geq 2 \left( 1+\frac{2}{n}-s \right) - (s^2-1).
\]
To verify \eqref{eq:mu_deriv_pos},
it thus suffices by $4k + 2n (s-1) \geq 0$ to show
\[
    (4k + 2 n (s-1)) \left( 2 \left( 1+\frac{2}{n}-s \right) - (s^2-1) \right)
    + \frac{2}{n} \left( 2 n (s-1) (3ks+k+2)-8ks+n^2 (s-1)^3 \right)
    > 0.
\]
A tedious, but not difficult computation shows that
the left-hand side can be simplified to
\[
    8 (n-k) \left( 1+\frac{2}{n}-s \right) (s - 1),
\]
which is indeed positive for $s \in (1,1+\frac{2}{n})$.
Thus, $\mu_{n,k}$ is strictly increasing.
A direct computation finally shows
$\zeta_{n,k}(1) = \frac{4}{n}$,
$\mu_{n,k}(1) = \sqrt{n}$,
$\zeta_{n,k}(1+\frac{2}{n}) = \frac{4}{n} + \frac{4}{n^2}$,
and $\mu_{n,k}(1+\frac{2}{n}) = \sqrt{n+1}$,
so (i) is established.

For (ii),
it is clear that $\tau_{n,k}$ is differentiable in
$\mu \in (n, n + 1)$ with derivative
\[
    \frac{
        \left(
            \frac{k}{2 \sqrt{\mu-n}}
            + \frac{2 (k+1) \mu - k (n+1)}{2 \sqrt{(k(\mu-n)+\mu-k) \mu}}
        \right) (k (\mu-n)+\mu)
        - (k+1) \left(
            k \sqrt{\mu-n}
            + \sqrt{(k(\mu-n)+\mu-k) \mu}
        \right)
    }
    {(k (\mu-n)+\mu)^2}.
\]
We claim that this term is positive.
To prove this,
we multiply the numerator of the derivative with
$\frac{2}{k} \sqrt{\mu-n} \, \sqrt{(k(\mu-n)+\mu-k) \mu}$
such that it can be written as $g^+(\mu) - g^-(\mu)$ with
\begin{align*}
    g^+(\mu)
    & \coloneqq \sqrt{(k (\mu-n)+\mu-k) \mu} \,
    ((k (\mu-n) + \mu) - 2 (k+1) (\mu-n))
    \\
    & = \sqrt{(k (\mu-n)+\mu-k) \mu} \,
    (n (k+2) - (k+1) \mu)
\intertext{
and
}
    g^-(\mu)
    & \coloneqq \sqrt{\mu-n}
    \frac{(2 (k+1) \mu - k(n+1))(k(\mu-n)+\mu)
    - 2 (k+1) \mu (k(\mu-n)+\mu-k)}{k}
    \\
    & = \sqrt{\mu-n}
    \frac{2 (k+1) \mu k - k (n+1)(k(\mu-n)+\mu)}{k}
    \\
    & = \sqrt{\mu-n} \,
    (k n (n+1) - (n-1)(k+1) \mu).
\end{align*}
Since $\mu < n+1$ and
$n(k+2) - (k+1)(n+1) = n - (k+1) \geq 0$,
we have $g^+(\mu) > 0$.
Thus, it suffices to prove that $g^+(\mu)^2 > g^-(\mu)^2$
to obtain $g^+(\mu) - g^-(\mu) > 0$.
It is possible to factorize
\[
    g^+(\mu)^2 - g^-(\mu)^2
    = (n + 1 - \mu) (n (n+1) - (k+1) \mu) ((k+1) \mu - n k)^2.
\]
By $n < \mu < n+1$ and $k+1 \leq n$,
we conclude that this expression is indeed positive.
Altogether,
it follows that $\tau_{n,k}$ is strictly increasing.
A direct computation of $\tau_{n,k}(n)$ and $\tau_{n,k}(n+1)$ completes the proof of (ii).

For (iii), let $\mu \in [n,n+1]$.
We observe that $p_\mu \colon \R \to \R$,
$p_\mu(t) = (t \sqrt{\mu-n} - 1)^2 - \mu \frac{1-t^2}{k}$
is a polynomial of degree $2$ in $t$ with zeros
\begin{equation} \label{eq:p_zeros}
    t_\mu^\pm
    \coloneqq \frac{
            \sqrt{\mu-n}
            \pm \sqrt{
                (\mu-n)
                - (\mu-n+\frac{\mu}{k})
                    (1-\frac{\mu}{k})
            }
        }{(\mu-n+\frac{\mu}{k})}
    = \frac{
        k \sqrt{\mu-n}
        \pm \sqrt{(k (\mu-n) + \mu - k) \mu}
    }{k (\mu - n) + \mu}.
\end{equation}
We note that $t_\mu^+ = \tau_{n,k}(\mu) > 0$,
and $t_\mu^- < 0$ by
\[
    \sqrt{(k(\mu-n)+\mu-k) \mu}
    > \sqrt{k (\mu-n) \mu}
    \geq k \sqrt{\mu-n}.
\]

Lastly, we turn to (iv) and let $s \in [1,1+\frac{2}{n}]$.
Since $\mu_{n,k}(1) = n$,
the claim holds if $s=1$.
For $s \neq 1$, we have for any $\mu \in [n,n+1]$ that
\begin{equation} \label{eq:p_q_relation}
    p_\mu \left( \frac{2 \sqrt{\mu-n}}{n (s-1)} \right)
    = \left( \frac{2 (\mu-n)}{n (s-1)} - 1 \right)^2
    - \mu \frac{n^2(s-1)^2 - 4(\mu-n)}{k n^2 (s-1)^2}
    = \frac{q(\mu)}{k n^2 (s-1)^2},
\end{equation}
where $q \colon \R \to \R$
is for given $n$, $k$, and $s$
a polynomial of degree $2$ defined by
\[
    q(t)
    = k (2t - n(s+1))^2 + t (4 (t-n) - n^2 (s - 1)^2).
\]
It is straightforward to verify that the zeros of $q$ are
\[
    t_q^\pm
    \coloneqq \frac{n}{8 (k+1)}
    \left(
        n (s-1)^2 + 4 k (s+1) + 4
        \pm \sqrt{
            (n (s-1)^2 + 4 k (s+1) + 4)^2
            - 16 (k+1) k (s+1)^2
        }
    \right),
\]
where a tedious, but not difficult computation shows
\[
    (n (s-1)^2 + 4 k (s+1) + 4)^2 - 16 (k+1) k (s+1)^2
    = n^2 \zeta_{n,k}(s)^2.
\]
Now, $\mu_{n,k}(s) = t_q^+$ together with \eqref{eq:p_q_relation} yields
\[
    p_{\mu_{n,k}(s)} \left( \frac{2 \sqrt{\mu_{n,k}(s)-n}}{n (s-1)} \right)
    = \frac{q(\mu_{n,k}(s))}{k n^2 (s-1)^2}
    = 0.
\]
This shows that $\frac{2 \sqrt{\mu_{n,k}(s)-n}}{n (s-1)}$
is a non-negative zero of $p_{\mu_{n,k}(s)}$,
so (iii) shows
\[
    \frac{2 \sqrt{\mu_{n,k}(s)-n}}{n (s-1)}
    = t_{\mu_{n,k}(s)}^+
    = \tau_{n,k}(\mu_{n,k}(s)).
\]
Altogether, (iv) is established.
\end{proof}

\end{document}